\documentclass[11pt, reqno]{amsart}
\usepackage{amsmath,amsfonts,amsthm,amssymb,latexsym}
\usepackage[a4paper,margin=3cm]{geometry}
\usepackage{graphicx}
\usepackage{url}
\usepackage[latin1]{inputenc}
\usepackage[T1]{fontenc} 
\usepackage{graphicx}
\usepackage[frenchb,english]{babel}
\usepackage{calc}
\usepackage{amsmath,amsthm}
\usepackage{url}
\usepackage{times, amssymb, amscd, mathrsfs, graphicx, color}  
\usepackage{enumerate}
\usepackage{hyperref}

\definecolor{darkred}{rgb}{0.9,0.1,0.1}

\newtheorem{theo}{Theorem}[section]

\newtheorem{coro}[theo]{Corollary}
\newtheorem{lem}[theo]{Lemma}

\newtheorem{Rq}[theo]{Remark}

\theoremstyle{plain}
\newtheorem*{hypo*}{Assumption}

\newcommand{\un}{\mathbf{1}} 
\newcommand{\N}{\mathbb{N}}                                              
\newcommand{\R}{\mathbb{R}}                                              

\newcommand{\p}{\mathbb{P}}  
\newcommand{\E}{\mathbb{E}}                                            

\newcommand{\var}{\textrm{Var}}  
\newcommand{\cov}{\textrm{Cov}}

\title[Fleming-Viot particle system and quasi-stationary distributions]{Quantitative results for the Fleming-Viot particle system and quasi-stationary distributions in discrete space} %

\date{\today}

\author{Bertrand~\textsc{Cloez}}  %
\address[Bertrand~\textsc{Cloez}]{Universit\'e de Toulouse, Institut de Math\'ematiques de Toulouse, CNRS UMR5219, France}%
\email{\url{bertrand.cloez@univ-toulouse.fr}}%

\author{Marie-No\'{e}mie~\textsc{Thai}} %
\address[Marie-No\'{e}mie~\textsc{Thai}]{Laboratoire d'Analyse et de Math\'ematiques
  Appliqu\'ees, CNRS UMR8050, Universit\'e Paris-Est Marne-la-Vall\'ee,
  France} %
\email{\url{marie-noemie.thai@univ-paris-est.fr}}

\thanks{Partially supported by ANR-11-LABX-0040-CIMI within the program ANR-11-IDEX-0002-02}

\begin{document}

\maketitle

\begin{abstract} 
We show, for a class of discrete Fleming-Viot (or Moran) type particle systems, that the convergence to the equilibrium is exponential for a suitable Wassertein coupling distance. The approach provides an explicit quantitative estimate on the rate of convergence. As a consequence, we show that the conditioned process converges exponentially fast to a unique quasi-stationary distribution. Moreover, by estimating the two-particle correlations, we prove that the Fleming-Viot process converges, uniformly in time, to the conditioned process with an explicit rate of convergence. We illustrate our results on the examples of the complete graph and of $N$ particles jumping on two points. 
\end{abstract}
{\footnotesize \textbf{AMS 2000 Mathematical Subject Classification:} 60K35, 60B10, 37A25.}


\renewcommand*\abstractname{ \ }
\begin{abstract}
\tableofcontents
\end{abstract}

\section{Introduction}

This paper deals with a (time-continuous) Moran type model, referred to as the Fleming-Viot process in the literature \cite{BHM00,FM07}, which approximates Markov semigroup conditioned on non-absorption. Briefly, when considering a time-continuous Markov chain, an interesting question is about the quasi-stationary distribution of the process which is killed at some rate, see for instance \cite{CMM,MV12}. Instead of conditioning on non-killing, it is possible to start $N$ copies of the Markov chain and, instead of being killed, one chain jumps randomly on the state of another one. The resulting process is a version of the Moran model that we will call Fleming-Viot. While the convergence of the large-population limit of the Moran model to the quasi-stationary distribution was already shown under some assumptions\cite{MM00,FM07,V11}, the present paper is concerned with deriving bounds for the rate of convergence. Our first main result, namely Theorem \ref{th:tl-general}, establishes the exponential ergodicity of the particle system with an explicit rate. This seems to be a novelty. As a consequence, we prove that the correlations between particles vanish uniformly in time, see Theorem \ref{thm:cor} and Theorem \ref{th:cor-gen}. This is also a new result even if \cite{FM07} gives a similar bound heavily depending on time. As application, we also give new proofs for some more classical but important results as a rate of convergence as $N$ tends to infinity (Theorem \ref{th:chaos}) which can be compared to the results of \cite{MM00,GJ12,V11}, a quantitative convergence  of the conditioned semi-group (Corollary \ref{coro:CV-QSD}) comparable to the results of \cite{MM02,MMV12} and uniform bound (in time) as $N$ tends to infinity (see Corollary \ref{coro: IndTime}), which seems to be new in discrete space but already proven for diffusion processes in \cite{MR06} with an approach based on martingale inequality and spectral theory associated to Schr\"{o}dinger equation.

Let us now be more precise and introduce our model. Let $(Q_{i,j})_{i,j \in F^*}$ be the transition rate matrix of an irreducible and positive recurrent continuous time Markov process on a  discrete and countable state space $F^*$. Set $F=F^* \cup \{ 0 \}$ where $0 \notin F^{*}$ and let $p_0 : F^* \mapsto \R_+$ be a non-null function. The generator of the Markov process $(X_t)_{t\geq 0}$, with transition rate $Q$ and death rate $p_0$, when applied to bounded functions $f:F \mapsto \R$, gives
$$
G f(i) = p_0(i)(f(0)-f(i)) + \sum_{j \in F^{*}} Q_{i,j} (f(j)-f(i)),
$$
for every $i\in F^*$ and $G f(0)=0$. If this process does not start from $0$ then it moves according to the transition rate $Q$ until it jumps to $0$ with rate $p_0$; the state $0$ is absorbing. Consider the process $(X_t)_{t\geq 0}$ generated by $G$ with initial law $\mu$ and denote by $\mu T_t$ its law at time $t$ conditioned on non absorption up to time $t$. That is defined, for all non-negative function $f$ on $F^*$, by
$$
\mu T_t f = \frac{\mu P_t f}{ \mu P_t \un_{\{0\}^c} } = \frac{\sum_{y \in F^{*}} P_t f(y) \mu(y)}{\sum_{y \in F^{*}} P_t \un_{\{0\}^c}(y) \mu(y)},
$$
where $(P_t)_{t\geq 0}$ is the semigroup generated by $G$ and we use the convention $f(0)=0$. For every $x\in F^*$, $k\in F^*$ and  non-negative function $f$ on $F^*$, we also set 
$$T_tf(x) = \delta_x T_t f \ \text{ and } \ \mu T_t (k) = \mu T_t \un_{\{k\}}, \quad \forall t\geq 0.
$$
A quasi-stationary distribution (QSD) for $G$ is a probability measure $\nu_{\textrm{qs}}$ on $F^*$ satisfying, for every $t \geq 0$, $ \nu_{\textrm{qs}} T_t = \nu_{\textrm{qs}}$. The QSD are not well understood, nor easily amenable to simulation. To avoid these difficulties, Burdzy, Holyst, Ingerman, March \cite{BHM00}, and Del Moral, Guionnet, Miclo \cite{DMG,MM00} introduced, independently from each other, a Fleming-Viot or Moran type particle system. This model consists of finitely many particles, say $N$, moving in the finite set $F^*$. Particles are neither created nor destroyed. It is convenient to think of particles as being indistinguishable, and to consider the occupation number $\eta $ with, for $k \in F^*$ , $\eta(k)=\eta^{(N)}(k)$ representing the number of particles at site $k$. Each particle follows independent dynamics with the same law as $(X_t)_{t\geq 0}$ except when one of them hits state $0$; at this moment, this individual jumps to another particle chosen uniformly at random. The configuration $(\eta_t)_{t\geq0}$ is a Markov process with state space $E=E^{(N)}$ defined by
$$
E=\left\{ \eta : F^* \to \N \ | \ \sum_{i \in F^*}  \eta(i) =N \right\}.
$$
Applying its generator to a bounded function $f$ gives
\begin{equation}
\label{eq:generator2}
\mathcal{L}f(\eta)=\mathcal{L}^{(N)}f(\eta) = \sum_{i\in F^{*}} \eta(i) \left[ \sum_{j \in F^{*}} (f(T_{i \rightarrow j} \eta)-f(\eta)) \left(Q_{i,j} + p_{0}(i) \dfrac{\eta(j)}{N-1} \right)\right],
\end{equation}
for every $\eta \in E$, where, if $\eta(i) \neq 0$, the configuration $T_{i \rightarrow j} \eta$ is defined by 
$$
T_{i \rightarrow j} \eta(i) = \eta(i)-1, \ T_{i \rightarrow j} \eta(j) = \eta(j)+1, \ \text{ and } T_{i \rightarrow j} \eta(k) = \eta(k) \quad k \notin \{i,j\}.
$$
For $\eta \in E$, the associated empirical distribution $m(\eta)$ of the particle system is given by
$$
m(\eta)= \dfrac{1}{N}\sum_{k\in F^{*}} \eta(k) \delta_{\{k\}}.
$$
For $\varphi : F^* \to \R$ and $k\in F^*$, we also set $m(\eta)(\varphi)= \sum_{j \in F^*} \varphi(j) m(\eta)(\{j\})$ and $m(\eta)(k)=m(\eta)(\{k\})$. The aim of this work is to quantify (if they hold) the following limits:
\[ 
  \begin{array}{ c c c }
     m(\eta_t^{(N)}) & \underset{ t \rightarrow + \infty}{\stackrel{(a)}{\longrightarrow}} & m(\eta_\infty ^{(N)})\\
      \text{\small{(b)}} \big\downarrow  & \ &  \big\downarrow \text{\small{(c)}}\\     
     m(\eta_0) T_t & \underset{ t \rightarrow + \infty}{\stackrel{(d)}{ \longrightarrow }} & \nu_{\textrm{qs}}
  \end{array} 
\]
where all limits are in distribution and the limits $(b),(c)$ are taken as $N$ tends to infinity. More precisely, Theorem \ref{th:tl-general} gives a bound for the limit $(a)$, Theorem \ref{th:chaos} for the limit $(b)$, Corollary \ref{coro: IndTime} for the limit $(c)$ and finally Corollary \ref{coro:CV-QSD} for the limit $(d)$.

To illustrate our main results, we develop in detail the study of two examples. Those examples are very simple when you are interested by the study of $(T_t)_{t\geq 0}$ (QSD, rate of convergence ...) but there are important problems (and even some open questions) on the particle system (invariant distribution, rate of convergence...).
The first example concerns a random walk on the complete graph with sites $\{1,\dots, K\}$ and constant killing rate. Namely
$$
\quad \forall i,j\in \{1,\dots, K\}, i \neq j, \quad Q_{i,j} = \dfrac{1}{K},  \ p_0(i)=p >0.
$$
The quasi-stationary distribution is trivially the uniform distribution. However, the associated particle system does not behave as independent identically distributed copies of uniformly distributed particles and its behavior is less trivial. One interesting point of the complete graph approach is that it permits to reduce the difficulties of the Fleming-Viot to the interaction. Due to its simple geometry, several explicit formulas are obtained such as the invariant distribution, the correlations and the spectral gap. It seems to be new in the context of Fleming-Viot particle systems. 

The second example is the case where $F^*$ contains only two elements. The study of $(T_t)_{t\geq 0}$ is classically reduced to the study of a $2\times 2$ matrix. The study of the particle system, for its part, is reduced to the study of a birth-death process with quadratic rates. We are not able to find, even in the literature, a closed formula for its spectral gap. However, we give a lower bound not depending on the number of particles. The proofs are based on our main general theorem (coupling type argument) and a generalisation of \cite{M99} (Hardy's inequalities type argument). For this example, the only trivial limit to quantify is the limit $(d)$. The analysis of these two examples shows the subtlety of Fleming-Viot processes.

\subsection*{Long time behavior}
\label{sect:main-temps}
To bound the limit $(a)$, we introduce the parameter $\lambda$ defined by
$$
\lambda = \inf_{i,i'\in F^{*} }\left( Q_{i,i'} + Q_{i',i} + \sum_{j \neq i,i'} Q_{i,j} \wedge Q_{i',j} \right).
$$
This parameter controls the ergodicity of a Markov chain with transition rate $Q$ without killing. Note that $\lambda$ is slightly larger than the ergodic coefficient $\alpha$ defined in \cite{FM07} by:
$$
\alpha= \sum_{j \in F^{*}} \inf_{i \neq j} Q_{i,j}.
$$
In particular, if there exists $j\in F^{*}$ and $c>0$ such that for every $i\neq j$, $Q_{i,j} > c$ then $\lambda \geq c$. Before expressing our results, let us describe the different distances that we use. We endow $E$ with the distance $d_{1}$ defined, for all $\eta$, $\eta'\in E$, by
$$
d_{1}(\eta,\eta') = \frac{1}{2} \sum_{j\in F} \vert \eta(j) - \eta'(j)\vert,
$$
which is the total variation distance between $m(\eta)$ and $m(\eta')$ up to a factor $N$: $d_{1}(\eta,\eta')= N d_{\textrm{TV}}(m(\eta),m(\eta'))$. Indeed, recall that, for every two probability measures $\mu$ and $\mu'$, the total variation distance is given by 
\begin{align*}
d_{\textrm{TV}}(\mu,\mu')
&= \frac{1}{2} \sup_{\Vert f \Vert_\infty \leq 1} \left( \int f d\mu - \int f d\mu' \right)  = \inf_{\substack{X \sim \mu\\ X' \sim \mu'}} \p \left( X \neq X'\right),
\end{align*}
where the infimum runs over all the couples of random variables with marginal laws $\mu$ and $\mu'$. Now, if $\mu$ and $\mu'$ are two probability measures on $E$, the $d_1-$Wasserstein distance between these two laws is defined by
$$
\mathcal{W}_{d_{1}}(\mu, \mu') = \inf_{\substack{\eta \sim \mu\\ \eta' \sim \mu'}} \E\left[ d_{1}(\eta,\eta')\right],
$$
where the infimum runs again over all the couples of random variables with marginal laws $\mu$ and $\mu'$. The law of a random variable $X$ is denoted by $\mathcal{L}(X)$ and, along the paper, we assume that
$$
\sup(p_0)< \infty.
$$
Our first main result is:
\begin{theo}[Wasserstein exponential ergodicity]
\label{th:tl-general}
If  $\rho=  \lambda -(\sup(p_0)-\inf(p_0))$ then for any processes $(\eta_{t})_{t>0}$ and $(\eta'_{t})_{t>0}$ generated by \eqref{eq:generator2}, and for any $t \geq 0$, we have
$$
\mathcal{W}_{d_1} (\mathcal{L}(\eta_t), \mathcal{L}(\eta'_t)) \leq e^{-\rho t} \mathcal{W}_{d_1} (\mathcal{L}(\eta_0), \mathcal{L}(\eta'_0)).
$$
In particular, if $\rho>0$ then there exists a unique invariant distribution $\nu_{N}$ satisfying for every $t \geq 0$,
$$
\mathcal{W}_{d_1} (\mathcal{L}(\eta_t), \nu_{N}) \leq e^{-\rho t} \mathcal{W}_{d_1} (\mathcal{L}(\eta_0), \nu_{N}).
$$
\end{theo}

To our knowledge, it is the first theorem which establishes an exponential convergence for the Fleming-Viot particle system with an explicit rate. Note anyway that in \cite{V10}, it is shown that the particle system is exponentially ergodic, when the underlying dynamics follows a certain stochastic differential equation. Its proof is based on Foster-Lyapunov techniques \cite{MT93,HM11} and, contrary to us, the dependence on $N$ of the rates and bounds are unknown. So, this gives less informations.

 When the death rate $p_0$ is constant, our bound is optimal in terms of contraction. See for instance section \ref{sect:GC}, where the example of a random walk on the complete graph is developed. When the death rate is not constant, this bound is not optimal, for instance if the state space is finite, we can have $\rho<0$ even if the process can converge exponentially fast. Indeed, it can be an irreducible Markov process on a finite state space. Nevertheless, finding a general optimal bound is a difficult problem. See for instance Section \ref{sect:two}, where we study the case where $F^{*}$ contains only two elements. Even though in this case the study seems to be easy, we are not able to give a closed formula for the spectral gap (even if we give a lower bound in the general case). Also, note that the previous inequality is a contraction, this gives some information for small times and is more than a convergence result. Finally the previous convergence is stronger than a convergence in total variation distance as can be checked with Corollary \ref{coro:coalescent}.

\subsection*{Propagation of chaos}
\label{sect:main-chaos}
In general, two tagged particles in a large population of interacting ones behave in an almost independent way under some assumptions; see \cite{S89}. In our case, two particles are almost independent when $N$ is large and this gives the convergence of $(m(\eta_t))_{t\geq 0}$ to $(T_t)_{t\geq 0}$. 

To prove this result, we will assume that: 
\begin{hypo*}[boundedness assumption] \ 
\begin{enumerate}
\item[$(A)$]$\mathbf{Q_1} = \sup_{i\in F^*} \sum_{j\in F^* ,j \neq i} Q_{i,j} < + \infty \ \text{ and } \ \mathbf{p}= \sup_{i\in F^*} p_0(i) < + \infty$.
\end{enumerate}
\end{hypo*}
Under this assumption, the particle system converges to the conditioned semi-group. Moreover, when the state space is finite, this convergence is quantified in terms of total variation distance. To express this convergence, we set
$$
\E_{\eta}[f(X)] = \E[f(X) \mid \eta_0 = \eta],
$$
for every bounded function $f$, every $\eta \in E$ and every random variable $X$.

\begin{theo}[Convergence to the conditioned process]
\label{th:chaos}
Under Assumption $(A)$ and for $t\geq 0$, there exists $B,C>0$ such that, for all $\eta \in E$, and any probability measure $\mu$,  we have 
$$
\sup_{\Vert \varphi \Vert_{\infty} \leq 1 } \E_{\eta} \left[ \left|  m(\eta_t)(\varphi) - \mu T_t \varphi \right| \right] \leq C e^{Bt} \left( \dfrac{1}{\sqrt{N}} + d_{\textrm{TV}} (m(\eta), \mu) \right).
$$
All constants are explicit and detailed in the proof (In particular, they do not depend on $N$ and $t$).
\end{theo}

The proof is based on an estimation of correlations and on a Gronwall-type argument. More precisely our correlation estimate is given by:

\begin{theo}[Covariance estimates]
\label{thm:cor} 
Let $\rho$ be defined in Theorem \ref{th:tl-general}. Under Assumption $(A)$, we have for all $k,l \in F^{*}$, $\eta \in E$ and $t\geq 0$ 
$$
\left| \E_{\eta}\left[\frac{\eta_{t}(k)}{N}\frac{\eta_{t}(l)}{N} \right] - \E_{\eta}\left[\frac{\eta_{t}(k)}{N}\right] \E_{\eta}\left[\frac{\eta_{t}(l)}{N}\right] \right | 
\leq \frac{2(\mathbf{Q_1} +  \mathbf{p})}{N-1} \frac{1 - e^{- 2\rho t}}{\rho},
$$
with the convention $(1 - e^{- 2\rho t}) \rho^{-1} =2t$ when $\rho=0$.
\end{theo}
This theorem gives a decay of the variances and the covariances of the marginals of $\eta$. Actually, it does not give any information on the correlation but this slight abuse of language is used to be consistent with other previous works \cite{AFG,FM07}.

The previous theorem is a consequence of Theorem \ref{th:cor-gen} which gives some bounds on the correlations of more general functional of $\eta$. The proof of this result comes from a commutation relation between the \textit{carr\'{e} du champs} operator and the semigroup of $\eta$. This commutation-type relation gives a decay of the variance and thus, by the Cauchy-Schwarz inequality, of the correlations. The previous bound is uniform in time when $\rho>0$ and it generalizes several previous work \cite{AFG,FM07}. Indeed, as our proof differs completely to \cite{AFG,FM07} (proof based on a comparison with the voter model), we are able to use more complex functional of $\eta$ and our bounds are uniform in time. In particular, taking the limit $t\rightarrow + \infty$ when $\rho>0$, we have the decay of the correlations under the invariant distribution of $(\eta_t)_{t\geq 0}$. This seems to be new (in discrete or continuous state space).

Theorem \ref{th:chaos} is a generalization of \cite[Theorem 1.3]{AFG}, \cite[Theorem 1.2]{FM07} and of \cite[Theorem 2.2]{GJ12}. Our assumptions are weaker and our convergence estimate is in a stronger form. We can also cite \cite[Theorem 1.1]{MM00} and \cite[Theorem 1]{V11} which give the same kind of bound with a less explicit constant. However, these two theorems cover a more general setting. This theorem permits to extend the properties of the particle system to the conditioned process; see the next subsection. The proof of Theorem \ref{th:chaos} differs from all these theorems; it seems simpler and is only based on a Gronwall argument and the correlation estimates.

Finally, we can improve the previous bound in the special case of the complete graph random walk but, in general, we do not know how improve it even when card$(F^*)=2$; see Sections \ref{sect:GC} and \ref{sect:two}. As all constants are explicit, the previous theorem allows us to consider parameters depending on $N$ and to understand how the particle system evolves when $Q$ varies with the size of the population; see for instance Remark \ref{rq:dependenceGC}. 

\subsection*{Two main consequences}
\label{sect:main-csq}
We summarize two important consequences of our main theorems. Firstly, as $\rho$, defined in Theorem \ref{th:tl-general}, does not depend on $N$, we can take the limit as $N\rightarrow + \infty$ in Theorem \ref{th:tl-general}. This gives an \og easy-to-verify \fg \ criterion to prove the existence, uniqueness of a quasi-stationary distribution and the exponential convergence of the conditioned process to it. 

\begin{coro}[Convergence to the QSD]
\label{coro:CV-QSD}
Suppose that $\rho$ is positive and that Assumption $(A)$ holds. For any probability measure $\mu,\nu$, we have
\begin{equation}
\label{eq:contractionT}
\forall t\geq 0, \ d_{\text{\textrm{TV}}}\left(\mu T_t, \nu T_t \right) \leq e^{-\rho t} d_{\text{\textrm{TV}}}\left(\mu, \nu\right).
\end{equation}
In particular, there exists a unique quasi-stationary distribution $\nu_{\textrm{qs}}$ for $(T_t)_{t\geq 0}$ and for any probability measure $\mu$, we have
$$
\forall t\geq 0, \ d_{\textrm{TV}}\left(\mu T_t, \nu_{\textrm{qs}} \right) \leq e^{-\rho t}.
$$
\end{coro}

This corollary is closely related to several previous work \cite{DS67}, \cite[Theorem 1.1]{MM02}, \cite[Theorem 3]{MMV12} and \cite[Theorem 1.1]{FM07}. When $F$ is finite, the oldest result dates from 1967 \cite{DS67} where Darroch and Seneta give a similar bound without additional assumption. Nevertheless, the constants are less explicit because the proof is based on Perron-Frobenius Theorem. The other results are more recent. Under a slightly weaker condition, we recover \cite[Theorem 1.1]{FM07} in a stronger convergence and with an estimation of the rate of convergence. As in \cite[Theorem 1.1]{MM02}, a mixing condition for $Q$ and a regularity one for $p_0$ are assumed to obtain an exponential convergence to a QSD; namely, we assume that $\lambda$ is large enough and $(\sup(p_0)-\inf(p_0))$ is small enough. 
In \cite[Theorem 1.1]{MM02} they only need that $\sup(p_0)<+ \infty$ but, their mixing condition is stronger than ours. Finally \cite[Theorem 3]{MMV12} gives a weaker condition to obtain an exponential convergence with (generally) a lower and less explicit rate of convergence when our result applies. Also note that Assumption $(A)$ is not necessary; see Remark \ref{rq:cv-qsd}.

Our second corollary gives a uniform bound for the limits $(b)$ and $(c)$, namely the convergences as $N$ tends to infinity:
\begin{coro}[Uniform bounds]
\label{coro: IndTime}
If $\rho>0$, then under the assumptions of Theorem \ref{th:chaos}, there exist $K_0, \gamma >0$ such that , for every $\eta \in E$,
$$
\sup_{t\geq 0} \sup_{\Vert \varphi \Vert_{\infty} \leq 1} \E_{\eta} \left[ \vert m(\eta_t)(\varphi) - m(\eta) T_t\varphi\vert \right] \leq \dfrac {K_0} {N^{\gamma}}.
$$
All constants are explicit and given in \eqref{cor: cst}.
\end{coro}

In particular, if $\eta$ is distributed according to the invariant measure $\nu_{N}$, then under the assumptions of the previous corollary, there exist $K_0>0$ and $\gamma>0$ such that
$$
\E\left[ \vert m(\eta)(\varphi) - \nu_{\textrm{qs}}(\varphi) \vert \right] \leq \dfrac {K_0} {N^{\gamma}},
$$
for every $\varphi$ satisfying $\Vert \varphi \Vert_\infty \leq 1$. Namely, under its invariant distribution, the particle system converges to the QSD. Without rate of convergence, this limiting result was proved in \cite[Theorem 2]{AFG} when $F$ is finite. Whereas, here, a rate of convergence is given. To our knowledge, it is the first bound of convergence for this limit. Whenever  $F^*$ is finite, the conclusion of the previous corollary holds with a less explicit $\gamma$ even when $\rho \leq 0$; see Remark \ref{rq:borne-unif}. Note also that, closely related, article \cite{MR06} gives a similar result when the underlying dynamics is diffusive instead of discrete. Its approach is completly different and based on martingale properties and on spectral properties associated to Schr\"{o}dinger equation.
\\

The remainder of the paper is as follows. Section \ref{sect:proof} gives the proofs of our main theorems; Subsection \ref{sect:temps} contains the proof of Theorem \ref{th:tl-general}, Subsection \ref{sect:chaos} the proof of Theorem \ref{th:chaos} and the last subsection the proof of the corollaries.
We conclude the paper with Sections \ref{sect:GC} and \ref{sect:two}, where we give the two examples mentioned above. The first one illustrates the sharpness of our results. The study of the second one is reduced to a very simple process for which few properties are known. It illustrates the need of general theorems as those previously introduced.

\section{Proof of the main theorems}
\label{sect:proof}
In this section, we prove Theorems \ref{th:tl-general} and \ref{th:chaos} and the corollaries stated before.
Let us recall that the generator of the Fleming-Viot process with $N$ particles applied to bounded functions $f: E\rightarrow \mathbb{R}$ and $\eta \in E$, is given by
\begin{equation}
\label{eq:generatorFV}
\mathcal{L}f(\eta) = \sum_{i\in F^{*}}  \eta(i) \sum_{j\in F^{*}} \left( Q_{i,j}+ p_{0}(i) \frac{\eta(j)}{N-1}\right) \left( f(T_{i \rightarrow j} \eta)-f(\eta)\right).
\end{equation}
Now let us give two remarks about the dynamics of the Fleming-Viot particle system.

\begin{Rq}[Translation of the death rate]
\label{rq:invariance}
Let $(P_t)_{t\geq0}$ and $(P'_t)_{t\geq0}$ be two semi-groups with the same transition rate $Q$ but different death rates $p_0,p_0'$ and let  $(T_t)_{t\geq 0}, (T'_t)_{t\geq 0}$ be their corresponding conditioned semi-groups respectively. Using the fact that
$$
P_t \un_{\{0\}^c} =\E\left[e^{- \int_0^t p_0(X_s) ds}\right] \ \text{ and } \ P'_t \un_{\{0\}^c} =\E\left[e^{- \int_0^t p'_0(X'_s) ds}\right],
$$
for every $t\geq0$, it is easy to see that $(T_t)_{t\geq0}=(T'_t)_{t\geq0}$ as soon as $p_0-p'_0$ is constant. This invariance by translation is not conserved by the Fleming-Viot processes. The larger $p_0$ is, the more jumps are obtained and the larger the variance becomes. This is why our criterion about the existence of QSD does not depend on $\inf(p_0)$ and why our propagation of chaos result depends on it.
\end{Rq}

\begin{Rq}[Non-explosion]
The particle dynamics guarantees the existence of the process $(\eta_t)_{t\geq 0}$ under the condition that there is no explosion. In other words, our construction is global as long as the particles only jump finitely many times in any finite time interval. We naturally assume that the Markov process with transition $Q$ is not explosive but it is not enough for the existence of the particle system. Indeed, an example of explosive Fleming-Viot particle system can be found in \cite{BBP12}. However, the assumption that $p_0$ is bounded is trivially sufficient to guarantee this non-explosion.
\end{Rq}

\subsection{Proof of Theorem \ref{th:tl-general}}
\label{sect:temps}

\begin{proof}[Proof of Theorem \ref{th:tl-general}]
We build a coupling between two Fleming-Viot particle systems,  $(\eta_{t})_{t\geq0}$ and $(\eta'_{t})_{t\geq0}$, generated by \eqref{eq:generator2}, starting respectively from some random configurations $\eta_{0},\eta'_{0}$ in $E$. We will prove that they will be closer and closer. 

Let us begin by roughly describing our coupling and then be more precise. For every $t\geq 0$, we set $\xi(t)=\xi=(\xi_1, \dots, \xi_N) \in (F^{*})^{N}$ and $\xi'(t)=\xi'=(\xi_1', \dots, \xi_N')$ the respective positions of the $N$ particles of the two configurations $\eta_t$ and $\eta'_t$. Then 
$$
\forall i \in F^*, \ \eta_t(i) = \textrm{card}\{ 1 \leq k \leq N  \ | \ \xi_k =i  \} \ \text{ and } \ \eta'_t(i) = \textrm{card}\{ 1 \leq k \leq N  \ | \ \xi'_k =i  \}.
$$
Distance $d_1(\eta,\eta')$ represents the number of particles which are not in the same site; namely, changing the indexation,
$$
d_1(\eta,\eta') = \textrm{card}\{ 1 \leq k \leq N  \ | \ \xi_k \neq \xi'_k  \}. 
$$
We then couple our two processes in order to maximize the chance that two particles coalesce. 
In a first time, we forget the interaction; we have two systems of $N$ particles evolving independently from each others. If two particles are in the same site,  $\xi_k=\xi_k'$, then the Markov property entails that we can make them jump together. When two particles are not in the same site, we can choose our jumps time in such a way that one goes to the second one, with positive probability. These steps are represented by the jumps rate $A_Q$ below. 

Nevertheless, the situation is trickier when we consider the interaction. Indeed, let us now disregard the underlying dynamics and only regard the interaction. If two particles are in the same site, $\xi_k=\xi_k'$, then they have to be killed and jump over the other particles. If the empirical measures are the same $\eta=\eta'$ then we can couple the two particles in such a way they die at the same time (because they are in the same site) and jump in the same site (because the empirical measures are equal). If $\eta\neq \eta'$ then we can not do this but we can maximize the probability to coalesce. Indeed there is $N -d_1(\eta,\eta')$ particles which are in the same site and then a probability $(N -d_1(\eta,\eta'))/(N-1)$ to coalesce. If two particles are not in the same site, $ \xi_k \neq \xi_k'$ , we can try to kill one before the other and put it in the same site. This is also not always possible.

Before expressing precisely the jumps rates, let us give some explanations.  We call first configuration the particles represented by $\{ \xi_k \}$ and the second configuration the particles represented by $\{ \xi'_k \}$. We speak about couple of particles when there are two particles coming from different configurations. There is $\eta(i) =\textrm{card}\{ k \ | \ \xi_k=i \}$ particles on the site $i$ and we can write
$$
\eta(i) = (\eta(i) - \eta'(i))_+ + \eta(i) \wedge \eta'(i),
$$
where $(\cdot)_+ = \max(0, \cdot)$. The part $\eta(i) \wedge \eta'(i)$ represents the number of couples of particles on $i$ and $(\eta(i) - \eta'(i))_+$ the rest of particles coming from the first configuration. Note that
\begin{align*}
\sum_{i\in F^{*}} (\eta(i)-\eta'(i))_+
&= d_1(\eta, \eta') = N - \sum_{i\in F^{*}} \eta(i) \wedge \eta'(i). 
\end{align*}

Now, we describe in detail our coupling. It is Markovian and we describe it by expressing its generator and its jumps rate; for every bounded function $f$ and $\eta, \eta' \in E$, its generator $\mathbb{L}$ is given by
$$
\mathbb{L}f(\eta, \eta') = \sum_{i,i',j,j'\in F^{*}} A(i,i',j,j') (f(T_{i \rightarrow j}\eta,T_{i' \rightarrow j'} \eta')-f(\eta,\eta')),
$$
where we decompose the jump rate $A$ into two parts $A = A_{Q} + A_{p}$. The jumps rate $A_{Q}$, that depends only on the transition rate $Q$, corresponds to the jumps related to the underlying dynamics, namely it is the dynamics when a particle does not die. A Markov process having only $A_Q$ as jumps rate corresponds to a coupling of two systems of $N$ particles evolving independently from each others.
The jumps rate $A_{p}$, corresponds to the redistribution dynamics and depends only on $p_0$; it does not depend on the underlying dynamics but only on the interaction. The construction of $A_Q$ is then more classic and the construction of $A_p$ is new and specific to this interaction. In what follows, we give the expressions of $A_p$ and $A_Q$; the points $i,i',j,j'$ are always different in twos.
\begin{itemize}
\item There are $\eta(i)\wedge \eta'(i)$ couples of particles on site $i\in F^*$.
\begin{itemize}
\item For each couple,  both particles can jump to the same site $j \in F^*$, at the same time and through the underlying dynamics. This gives the following jumps rate:
$$
 A_Q(i,i,j,j) = \left( \eta(i)\wedge \eta'(i)\right)  Q_{i,j}.
$$
\item Both of them can die at the same time. With probability  $\frac{\eta (j)\wedge \eta'(j)}{N-1}$, they can jump to the same site $j$; this gives
$$
 A_p(i,i,j,j) = p_0 (i) \left( \eta(i)\wedge \eta'(i)\right) \dfrac{\eta (j)\wedge \eta'(j)}{N-1}.
$$
With probability $\frac{\eta (i)\wedge \eta'(i)-1}{N-1}$, both particles jump where they come from and, so, this changes anything. With probability
\begin{equation}
\label{eq:taux-sitediff}
\left( 1 -  \frac{ \sum_{k\in F^*}\eta (k)\wedge \eta'(k) -1}{N-1} \right)\frac{(\eta(j)-\eta'(j))_+}{ \sum_{k\in F^*} (\eta(k)-\eta'(k))_+ } \frac{(\eta'(j')-\eta(j'))_+}{ \sum_{k\in F^*} (\eta'(k)-\eta(k))_+ },
\end{equation}
they can jump to two different sites $j,j'$. Indeed, with probability  $1 -  \frac{\sum_{k\in F^*}\eta (k)\wedge \eta'(k) -1}{N-1} $, they can jump in different sites, and conditionally on this event, with probability  $\frac{(\eta(j)-\eta'(j))_+}{ \sum_{k\in F^*} (\eta'(k)-\eta(k))_+ }$, the first particle jumps in site $j$ and, with probability $\frac{(\eta'(j')-\eta(j'))_+}{ \sum_{k\in F^*} (\eta'(k)-\eta(k))_+ }$, the second one jumps in site $j'$. Probability \eqref{eq:taux-sitediff} is equal to 
$$
\dfrac{(\eta(j)- \eta'(j))_+  \cdot (\eta'(j')-\eta(j'))_+}{ (N-1) d_1(\eta,\eta')}.
$$
In short, this gives the following jump rates:

$$
A_p(i,i,j,j') = p_0(i) \left(\eta(i)\wedge \eta'(i)\right) \dfrac{(\eta(j)- \eta'(j))_+ \cdot (\eta'(j')-\eta(j'))_+}{ (N-1) d_1(\eta,\eta')}.
$$
\end{itemize}
\item For every site $i \in F^*$ there are $(\eta(i)- \eta'(i))_+$ particles from the first configuration which are not in a couple. For each of theses particles, we choose, uniformly at random, a particle of the second configuration (which is not coupled with another particle as in the first point). This particle, chosen at random, is on the site $i'\in F^*$ with probability 
$$
\frac{(\eta'(i')-\eta(i'))_+}{\sum_{k} (\eta'(k)-\eta(k))_+}=\frac{(\eta'(i')-\eta(i'))_+}{d_1(\eta,\eta')}.
$$
\begin{itemize}
\item 
For one of these new couple of particles coming from sites $i\neq i'$, both particles can jump at the same time to the same site $j$ (different from $i,i'$), through the underlying dynamics; this gives
$$
A_Q(i,i',j,j) = (\eta(i)- \eta'(i))_+ \cdot \frac{(\eta'(i')-\eta(i'))_+}{d_1(\eta,\eta')} \cdot  \left(Q_{i,j}\wedge Q_{i',j}\right). 
$$ 
Nevertheless, these two particles do not have the same jump rates (because they do not come from the same site), so it is possible that one jumps to another site while the other one does not jump (also through the underlying dynamics); this gives
$$
 	A_Q(i,i',j,i') 
 	= (\eta(i)- \eta'(i))_+  \cdot \frac{ (\eta'(i')-\eta(i'))_+}{d_1(\eta,\eta')} \cdot (Q_{i,j} - Q_{i',j})_+,
 	$$ 
 and
$$
A_Q(i,i',i,j') = (\eta(i)- \eta'(i))_+  \cdot \frac{(\eta'(i')-\eta(i'))_+}{d_1(\eta,\eta')} \cdot (Q_{i',j'} - Q_{i,j'})_+ .
$$
Also, one of them can jump to the site of the second one:
	$$
	A_Q(i,i',i',i') = \frac{(\eta(i)- \eta'(i))_+ \cdot (\eta'(i')-\eta(i'))_+}{d_1(\eta,\eta')}Q_{i,i'},
	$$
and
	$$
	A_Q(i,i',i,i) = \frac{(\eta(i)- \eta'(i))_+ \cdot (\eta'(i')-\eta(i'))_+}{d_1(\eta,\eta')}Q_{i',i}.
	$$
\item We focus now our attention on the redistribution dynamics. We would like that both particles of a couple die at the same time and jump to the same site $j$ (where a couple of particles exists; that is with probability $\frac{\eta (j) \wedge \eta'(j)}{N-1}$). This gives: 
$$
A_p(i,i',j,j)	= (\eta(i)- \eta'(i))_+ \cdot \frac{(\eta'(i')-\eta(i'))_+}{d_1(\eta,\eta')} \cdot \left(p_0(i) \wedge p_0(i')\right) \cdot\dfrac{\eta (j) \wedge \eta'(j)}{N-1}
$$ 
But, even if they die at same time, they can jump to different sites with rate 
\begin{align*}
A_p(i,i',j,j') 
&= \left(p_0(i) \wedge p_0(i')\right) \dfrac{(\eta(i)- \eta'(i))_+ (\eta'(i')-\eta(i'))_+}{ d_1(\eta,\eta')}\\
&\quad \cdot \dfrac{(\eta(j)- \eta'(j))_+ (\eta'(j')-\eta(j'))_+}{ (N-1) d_1(\eta,\eta')}.
\end{align*}
 
However, this is not always possible to kill them at the same time. If they do not then the dying particle jumps uniformly to a particle of its configuration; this gives
$$ 	
  A_p(i,i',j,i')= (\eta(i)- \eta'(i))_+ \cdot \frac{ (\eta'(i')-\eta(i'))_+}{d_1(\eta,\eta')} \cdot \left(p_0(i) - p_0(i')\right)_+ \cdot \dfrac{\eta(j)}{N-1},
$$
and
$$
 A_p(i,i',i,j') = (\eta(i)- \eta'(i))_+  \cdot \frac{(\eta'(i')-\eta(i'))_+}{d_1(\eta,\eta')} \cdot \left(p_0(i') - p_0(i) \right)_+ \cdot \dfrac{\eta(j')}{N-1}.
 $$
\end{itemize}
\end{itemize}

We set, for every measurable function $f$,
$$
\mathbb{L}_{Q}f(\eta, \eta') = \sum_{i,i',j,j'\in F^{*}} A_Q(i,i',j,j') (f(T_{i \rightarrow j}\eta,T_{i' \rightarrow j'} \eta')-f(\eta,\eta')),
$$
and 
$$
\mathbb{L}_{p}f(\eta, \eta') = \sum_{i,i',j,j'\in F^{*}} A_p(i,i',j,j') (f(T_{i \rightarrow j}\eta,T_{i' \rightarrow j'} \eta')-f(\eta,\eta')).
$$  
Our coupling is totally defined. It is a little bit long but not difficult to verify that if a measurable function $f$ on $E\times E$ does not depend on its first (resp. second) variable; that is with a slight abuse of notation:
$$
\forall \eta,\eta'\in E, \ f(\eta,\eta')=f(\eta) \ (\text{resp. } \ f(\eta,\eta')=f(\eta')),
$$
then $\mathbb{L} f(\eta,\eta')= \mathcal{L} f(\eta)$ (resp. $\mathbb{L} f(\eta,\eta')= \mathcal{L} f(\eta')$). This property ensures that the couple $(\eta_t,\eta'_t)_{t\geq 0}$ generated by $\mathbb{L}$  is well a coupling of processes generated by $\mathcal{L}$ (that is of Fleming-Viot processes). Now, let us prove that the distance between $\eta_t$ and $\eta'_t$ decreases exponentially. We have 
\begin{align*}
\mathbb{L}_{p} d_1(\eta,\eta') 
&\leq \sum_{i\in F^{*}} p_0(i) \left( \eta(i) \wedge \eta'(i)\right) \frac{d_1 (\eta,\eta')}{N-1} \\
&-\sum_{i,i'\in F^{*}} \left( p_0(i) \wedge p_0(i') \right) \frac{(\eta(i)-\eta'(i))_+(\eta'(i')-\eta(i'))_+}{d_1(\eta,\eta')} \sum_{j\in F^{*}} \frac{\eta(j) \wedge \eta'(j)}{N-1}\\
&\leq \left( \sup(p_0)- \inf(p_0) \right) \frac{d_1 (\eta,\eta')}{N-1} \left( N-d_1(\eta, \eta') \right)\\
&\leq (\sup(p_0) - \inf(p_0)) d_1(\eta,\eta').
\end{align*}
Now,
\begin{align*}
\mathbb{L}_{Q}d_1 (\eta,\eta') 
&\leq - \sum_{i,i'\in F^{*}} \left( Q_{i,i'} + Q_{i',i} + \sum_{j \neq i,i'} Q_{i,j} \wedge Q_{i',j} \right)  \frac{(\eta(i)-\eta'(i))_+(\eta'(i')-\eta(i'))_+}{d_1(\eta,\eta')} \\ 
& \leq - \lambda d_1(\eta, \eta').
\end{align*}

We deduce that 
$\mathbb{L}d_1 (\eta,\eta') \leq - \rho d_1(\eta,\eta').$ Now let $(\mathbb{P}_{t})_{t \geq 0}$ be the semi-group associated with the generator $\mathbb{L}$. Using the equality $\partial_{t} \mathbb{P}_{t}f = \mathbb{P}_{t}\mathbb{L}f$ and Gronwall Lemma, we have, for every $t\geq 0$, $\mathbb{P}_t d_1 \leq e^{- \rho t} d_1$; namely
$$
\E[d_1(\eta_t,\eta'_t)]
\leq e^{-\rho t} \E[d_1(\eta_0, \eta'_0)].
$$
Taking the infimum over all couples $(\eta_{0},\eta'_{0})$, the claim follows. The existence and the uniqueness of an invariant distribution come from classical arguments; see for instance \cite[Theorem  5.23]{C04}.
\end{proof}

As it is easy to see that the distance $\mathcal{W}_{d_1}$ is larger than the total variation distance, we have the following consequence:

\begin{coro}[Coalescent time estimate]
\label{coro:coalescent}
For all $t\geq0$, we have 
$$
d_{\textrm{TV}}(\mathcal{L}(\eta_{t}), \mathcal{L}(\eta'_{t})) \leq e^{-\rho  t} \mathcal{W}_{d_1} (\mathcal{L}(\eta_0), \mathcal{L}(\eta'_0)).
$$
In particular, if $\rho>0$ the invariant distribution $\nu_{N}$ satisfies
$$
d_{\textrm{TV}}(\mathcal{L}(\eta_{t}), \nu_{N}) \leq e^{-\rho  t} \mathcal{W}_{d_1} (\mathcal{L}(\eta_0), \nu_{N}).
$$
\end{coro}

The proof is simple and given for sake of completeness.
\begin{proof}
Using Theorem \ref{th:tl-general}, we find
\begin{align*}
d_{\textrm{TV}}(\mathcal{L}(\eta_{t}), \mathcal{L}(\eta'_{t}))&= \inf_{\substack{\eta_{t} \sim \mathcal{L}(\eta_{t})\\ \eta'_{t} \sim \mathcal{L}(\eta'_{t})}} \E\left[ \un_{\eta_{t} \neq \eta'_{t}}\right]\\
&\leq \inf_{\substack{\eta_{t} \sim \mathcal{L}(\eta_{t})\\ \eta'_{t} \sim \mathcal{L}(\eta'_{t})}} \E\left[d_1(\eta_{t} ,\eta'_{t})\right]= \mathcal{W}_{d_1} (\mathcal{L}(\eta_t), \mathcal{L}(\eta'_t))\\
& \leq e^{-\rho t} \mathcal{W}_{d_1} (\mathcal{L}(\eta_0), \mathcal{L}(\eta'_0)).
\end{align*}
\end{proof}

\begin{Rq}[Generalization]
\label{rq:generalization}
As we can see at the end of the paper, in the case where $F^*$ contains only two elements, the coupling that we use is pretty good but our estimation of the distance is (in general) too rough. There is some natural way to change the bound/criterion that we found. The first one is to use another more appropriate distance. This technique is in general useful in other (Markovian) contexts \cite{CJ12,CH13, E13}. Another way is to find a contraction after a certain time: it is the Foster-Lyapunov-type techniques \cite{BCG,HM11,MT93}. This type of techniques give more general criteria but are useless for small times and the formulas we get are less explicit. All of these techniques will give different criteria that are not necessarily better. Finally note that, in all the paper, we can replace $\rho$ by
$$
\rho'= \inf_{i,i' \in F^*} \left\{ p_0(i) \wedge p_0(i') + Q_{i,i'} + Q_{i',i} + \sum_{j \neq i,i'} Q_{i,j} \wedge Q_{i',j} \right\} - \sup(p_0),
$$
and all conclusions hold. Indeed, we have to bound directly $\mathbb{L} d_1$ instead of bounding separately  $\mathbb{L}_Q d_1$ and  $\mathbb{L}_p d_1$.
\end{Rq}

\subsection{Proofs of Theorems \ref{th:chaos} and \ref{thm:cor}}
\label{sect:chaos}

The proof of Theorem \ref{th:chaos} is done in two steps. Firstly, we estimate the correlations between the number of particles over the sites and then we estimate the distance in total variation via the Kolmogorov equation. Let us introduce some notations. For every bounded functions $f,g$, every $\eta \in E$ and every random variable $X$, we set
$$
\cov_{\eta}[f(X), g(X)] = \E_{\eta}[f(X)g(X)] - \E_{\eta}[f(X)] \E_{\eta}[g(X)],
$$
and 
$$\var_{\eta}[f(X)] = \cov_{\eta}[f(X), f(X)].
$$

Let $(S_t)_{t\geq 0}$ be the semigroup of $(\eta_t)_{t\geq 0}$ defined by
$$
S_t f(\eta) = \E_\eta[ f(\eta_t)],
$$
for every $t\geq0$, $\eta\in E$ and bounded function $f$. If $\mu$ is a probability measure on $E$ and $t\geq 0$, then $\mu S_t$ is the measure defined by
$$
\mu S_t f = \int_E S_t f(y) \mu(dy).
$$
It represents the law of $\eta_t$ when $\eta_0$ is distributed according to $\mu$.  We also introduce the \textit{carr\'e du champ} operator $\Gamma$ defined, for any bounded function $f$ and $\eta\in E$, by
\begin{align}
\Gamma f(\eta) &= \mathcal{L}(f^2)(\eta) - 2 f(\eta) \mathcal{L} f(\eta)\label{eq:carre-du-champ}\\
&= \sum_{i,j\in F^{*}} \eta(i) \left(Q_{i,j} + p_0(i) \frac{\eta(j)}{N-1} \right) \left(f( T_{i \rightarrow j} \eta) - f(\eta) \right)^2.\nonumber
\end{align}

We present now an improvement of Theorem \ref{thm:cor}.
\begin{theo}[Correlations for Lipschitz functional]
\label{th:cor-gen} 
Let $g,h$ be two $1$-Lipschitz mappings on $(E,d_1)$; namely
$$
|g(\eta) - g(\eta')| \leq d_1(\eta,\eta') \ \text{ and } \ |h(\eta) - h(\eta')| \leq d_1(\eta,\eta') ,
$$
for every $\eta,\eta' \in E$. 
Under Assumption $(A)$ we have for all $t\geq 0$ and $\eta \in E$,
$$
\left| \cov_{\eta}(g(\eta_t),h(\eta_t))\right| \leq \frac{1 - e^{- 2\rho t}}{2\rho} \left(N \mathbf{Q_{1}} +  \mathbf{p}\frac{N^2}{N-1}\right),
$$
with the convention $(1 - e^{- 2\rho t}) \rho^{-1} =2t$ when $\rho=0$.
\end{theo}
In particular, if $\rho>0$ then the previous bound is uniform.

\begin{proof}
For any function $g$ on $E$ and $t\geq 0$, we have
$$
\var_{\eta} (g(\eta_t)) =  S_t(g^2)(\eta)-(S_t g)^2(\eta) = \int_0^t S_s \Gamma S_{t-s}g (\eta) ds.
$$

Indeed, setting, for any $s\in [0,t]$ and $\eta \in E$, $\Psi_{\eta}(s) = S_s \left[ (S_{t-s}g)^{2}\right](\eta)$ and $\psi(s)=S_{t-s}g$, we get
$$
\forall s\geq 0, \ \Psi_{\eta}'(s) = S_s \left[ \mathcal{L}\psi^{2} - 2 \psi \mathcal{L}\psi\right](\eta) = S_s \Gamma \psi(s)(\eta),
$$
and so,
\begin{align*}
\var_{\eta} (g(\eta_t))&= \Psi_{\eta}(t) -\Psi_{\eta}(0)= \int_0^t S_s \Gamma S_{t-s}g (\eta) ds.
\end{align*}
Now, if $g$ is a $1$-Lipschitz mapping with respect to $d_1$ then
\begin{align*}
\mid S_{t-s}g(T_{i \rightarrow j} \eta) - S_{t-s}g(\eta)\mid
&\leq \E\left[|g( \eta'_{t-s}) - g(\eta_{t-s})|\right] \leq  \E\left[d_{1}(\eta_{t-s}, \eta'_{t-s})\right],
\end{align*}
where $\eta_{t-s}, \eta'_{t-s}$ evolve as Fleming-Viot particle systems with initial conditions $\eta$ and $T_{i \rightarrow j} \eta$. Thus, using Theorem \ref{th:tl-general}, we obtain
\begin{align}
\mid S_{t-s} g(T_{i \rightarrow j} \eta) - S_{t-s}g(\eta)\mid 
&\leq \mathcal{W}_{d_1} (\mathcal{L}(\eta_{t-s}), \mathcal{L}(\eta'_{t-s})) \label{eq:KR}\\
&\leq  e^{- \rho (t-s)} d_1(T_{i \rightarrow j} \eta, \eta)\nonumber\\
&\leq e^{- \rho (t-s)}\un_{i \neq j}.\nonumber
\end{align}
Hence,
$$
\Vert \Gamma S_{t-s} g \Vert_ \infty = \sup_{\eta\in E} \mid \Gamma S_{t-s} g (\eta)\mid   \leq e^{- 2\rho (t-s)} \left(N \mathbf{Q_{1}} + \mathbf{p} \frac{N^2}{N-1} \right).
$$

Finally, the Cauchy-Schwarz inequality and the first part of the proof give
\begin{align*}
\vert \cov_{\eta}(g(\eta_{t}), h(\eta_{t})) \vert 
&\leq  \var_{\eta}(g(\eta_{t}))^{1/2}\var_{\eta}(h(\eta_{t}))^{1/2}\\
&\leq \frac{1 - e^{- 2\rho t}}{2\rho} \left(N \mathbf{Q_{1}} +  \mathbf{p}\frac{N^2}{N-1} \right).
\end{align*}
\end{proof}


\begin{proof}[Proof of Theorem \ref{thm:cor}]
Fix $l\in F^*$ and set $\varphi_l :\eta \mapsto \eta(l)$. The function $ \varphi_l/2$ is a $1$-Lipschitz mapping with respect to $d_1$, so we apply the previous theorem .
\end{proof}

\begin{Rq}[Generalization]
\label{cor-amel}
Assume that there exist $C>0$ and $\lambda>0$ such that for any processes $(\eta_{t})_{t>0}$ and $(\eta'_{t})_{t>0}$ generated by \eqref{eq:generator2}, and for any $t > 0$, we have
\begin{equation}
\label{eq:cv-FV}
\mathcal{W}_{d_1} (\mathcal{L}(\eta_t), \mathcal{L}(\eta'_t)) \leq C e^{-\lambda t} \mathcal{W}_{d_1} (\mathcal{L}(\eta_0), \mathcal{L}(\eta'_0)),
\end{equation}
then under the previous assumptions we have, for all $t\geq 0$,
$$
\cov_{\eta}(\eta_{t}(k)/N,\eta_{t}(l)/N) \leq \frac{2 C}{N^2} \frac{1 - e^{- 2\lambda t}}{\lambda} \left(N \mathbf{Q_{1}} +  \mathbf{p}\frac{N^2}{N-1}\right).
$$
A bound like \eqref{eq:cv-FV} is proved when the state space $F^*$ contains only two points. 
\end{Rq}

\begin{proof}[Proof of Theorem \ref{th:chaos}]

The proof is based on a bias-variance type decomposition. The variance is bounded through Theorem \ref{thm:cor} and the bias through Gronwall-type argument. More precisely, for $t \geq 0$, we have
\begin{equation}
\label{eq:BV}
\sup_{\Vert \varphi \Vert_\infty \leq 1} \E_{\eta} \left[\left| m(\eta_t)(\varphi) - \mu T_{t} \varphi \right| \right]  \leq \sup_{\Vert \varphi \Vert_\infty \leq 1} \E_{\eta} \left[ \left| m(\eta_t)(\varphi) - \overline{m}(\eta_t)(\varphi) \right| \right] + 2 d_{\textrm{TV}}( \overline{m}(\eta_t), \mu T_{t} ),
\end{equation}
where $\overline{m}(\eta_t)$ is the empirical mean measure; namely $\overline{m}(\eta_t) (k) =\E[m(\eta_t)(k)]$, for every $k\in F^*$. Let $\varphi$ be a function such that $\Vert \varphi \Vert_\infty \leq 1$. Cauchy-Schwarz inequality gives
\begin{align*}
\E_{\eta} \left[\left | m(\eta_t)(\varphi) - \overline{m}(\eta_t)(\varphi) \right| \right] \leq 
2N^{-1}\var(g_{\varphi}(\eta_t))^{1/2},
\end{align*}
where $g_{\varphi}: \eta\mapsto \frac{1}{2}\sum_{k \in F^{*}} \eta(k) \varphi(k)= \frac{N}{2} m(\eta)(\varphi)$ is a 1-Lipschitz function. So by  Theorem \ref{th:cor-gen} we have
\begin{align*}
\sup_{\Vert \varphi \Vert_\infty \leq 1} \E_{\eta} \left[ \left| m(\eta_t)(\varphi) - \overline{m}(\eta_t)(\varphi) \right| \right] \leq \sqrt{2\rho^{-1}(1-e^{-2\rho t})(\mathbf{Q_{1}} + \mathbf{p}) (N-1)^{-1}}.
\end{align*}

Now, to study the bias term in \eqref{eq:BV}, let us introduce the following notations
$$
u_{k}(t) = \E_{\eta}[m(\eta_t)(k)] \ \text{ and } \ v_{k}(t) = \mu T_{t}(k).
$$
It is well known that $(\mu T_t)_{t\geq 0}$ is the unique measure solution to the (non-linear) Kolmogorov forward type equations: $\mu T_0 = \mu$, and
\begin{equation}
\label{eq:kolmo}
\forall t\geq 0,  \ \partial_t \mu T_t(j) =  \sum_{i\in F^{*}} \left( Q_{i,j} \ \mu T_t(i) + \  p_0(i) \ \mu T_t(i) \ \mu T_t (j)\right).
\end{equation}
Thus
$$
\partial_{t} v_{k}(t) = \sum_{i\in F^{*}} Q_{i,k} v_{i}(t) + \sum_{i\in F^{*}} p_{0}(i) v_{i}(t)v_{k}(t).
$$
Also, $u_{k}(t)= \E_{\eta}[m(\eta_t)(k)]= S_t f(\eta)$, where $f :\eta\mapsto m(\eta)(k)$ and $(S_t)_{t\geq 0}$ is the semi-group of $(\eta_t)_{t\geq 0}$, thus, using \eqref{eq:generator2}, the equality $\partial_t S_tf = \mathcal{L} S_tf$ and the convention that $p_0(i) + \sum_{j \in F^{*}} Q_{i,j} = 0$ for every $i \in F^{*}$, we find 
$$
\partial_{t} u_{k}(t) = \sum_{i\in F^{*}} Q_{i,k} u_{i}(t) + \sum_{i\in F^{*}} p_{0}(i) u_{i}(t)u_{k}(t) - \dfrac{p_{0}(k)}{N-1} u_{k}(t) + R_{k}(t),
$$
where 
\begin{align*}
R_{k}(t)  
&= \sum_{i\in F^{*}} p_{0}(i)\left(\dfrac{N}{N-1} \E_{\eta}(m(\eta_t)(i)m(\eta_t)(k)) - \E_{\eta}(m(\eta_t)(i))\E_{\eta}(m(\eta_t)(k))\right)\\
&=  \E_{\eta}\left(\left(\sum_{i\in F^{*}} p_{0}(i) m(\eta_t)(i) \right)m(\eta_t)(k)\right) - \E_{\eta}\left(\sum_{i\in F^{*}} p_{0}(i)m(\eta_t)(i)\right) \E_{\eta} \left(m(\eta_t)(k) \right) \\
&+ (N-1)^{-1} \E_{\eta}\left(\left(\sum_{i\in F^{*}} p_{0}(i) m(\eta_t)(i) \right) m(\eta_t)(k)\right).
\end{align*}

For $t \geq 0$, let us define $\epsilon(t) = \sum_{k \in F^{*}} \vert u_k(t) - v_k(t) \vert= 2 d_{\textrm{TV}}( \overline{m}(\eta_t), \mu T_{t})$. Using triangular inequality, Fubini-Tonelli Theorem and Assumption $(A)$, we have

\begin{align*}
\epsilon(t)
&= \sum_{k \in F^{*}} \left| u_k(0) -v_k(0) + \int_0^t \partial_s (u_k(s) -v_k(s)) ds \right|\\
&\leq \epsilon(0) + \sum_{k \in F^{*}} \int_0^t \left| \sum_{i\in F^{*}} Q_{i,k} (u_i(s) - v_{i}(s)) \right| + \sum_{k \in F^{*}} \int_0^t \left( \dfrac{p_{0}(k)}{N-1} u_{k}(s) +  \left|R_{k}(s)\right| \right) ds \\
&\quad + \sum_{k \in F^{*}} \int_0^t \left| \sum_{i\in F^{*}} p_{0}(i) \left[ v_{i}(s)(u_k(s) - v_k(s)) + u_{k}(s) (u_i(s) - v_{i}(s)) \right]\right| ds\\
&\leq \epsilon(0) +  \int_0^t \left(\mathbf{Q_{1}} +2\mathbf{p}\right)\epsilon(s) ds + \dfrac{\mathbf{p}t}{N-1} + \int_0^t \sum_{k \in F^{*}} |R_{k}(s)| ds.
\end{align*}
However, by Cauchy-Schwarz inequality and Theorem \ref{th:cor-gen} with the $1$-Lipschitz function $g: \eta \mapsto \frac{1}{2\mathbf{p}} \sum_{i \in F^*} p_0(i) \eta(i)$, we have
\begin{align*}
\sum_{k\in F^*}|R_{k}(t)|
&\leq \sum_{k\in F^*} \E_{\eta} \left( m(\eta_t)(k) \left| \sum_{i \in F^*} p_0(i) m(\eta_t) (i) - \E_{\eta} \left(\sum_{i \in F^*} p_0(i) m(\eta_t) (i) \right) \right| \right) + \mathbf{p} (N-1)^{-1}\\ 
&\leq 2 \mathbf{p} N^{-1} \var_{\eta}(g(\eta_t))^{\frac{1}{2}}  + \mathbf{p} (N-1)^{-1}\\ 
&\leq \mathbf{p} \sqrt{2\rho^{-1}(1-e^{-2\rho t})(\mathbf{Q_{1}} +\mathbf{p})(N-1)^{-1}} + \mathbf{p}(N-1)^{-1}.
\end{align*}

If $c_t =\rho^{-1}(1-e^{-2\rho t})$, $B=\mathbf{Q_{1}} +2\mathbf{p}$ then Gronwall's lemma gives
\begin{align}
\varepsilon(t) 
&\leq \varepsilon(0) e^{B t} + \int_0^t e^{B(t-s)} \left(\frac{2 \mathbf{p} \sqrt{2 B}}{(N-1)^{1/2}} \sqrt{c_s} + \frac{2\mathbf{p}}{N-1}\right) ds \nonumber \\
&\leq  \left(\varepsilon(0) + \frac{2\mathbf{p}}{(N-1)B} + \frac{2 \mathbf{p} \sqrt{2 B}}{(N-1)^{1/2}} \int_0^t e^{- B s}\sqrt{c_s} ds \right) e^{B t} \nonumber \\
&\leq \left(\varepsilon(0) + \frac{A}{\sqrt{N}}\right)e^{Bt} \nonumber,
\end{align}
for some $A>0$.

\end{proof}

\subsection{Proof of the corollaries}

In this subsection, we give the proofs of corollaries given in the introduction.
\begin{proof}[Proof of Corollary \ref{coro:CV-QSD}]

The proof is based on an approximation of the conditioned semigroups by two particle systems. Theorem \ref{th:tl-general} gives a contraction for these particle systems. We then use Theorem \ref{th:chaos} and a discretization argument to prove that it implies a contraction for the conditioned semigroups.

Let $(m_0^{(N)})_{N \geq 0}$ and $(\widetilde{m}_0^{(N)})_{N \geq 0}$ be two sequences of probability measures that converge to $\mu$ and $\nu$ respectively , as $N$ tends to infinity, and such that $\eta_0^{(N)} = (N m_0^{(N)}(k))_{k\in F^*}\in E^{(N)}$ and $\widetilde{\eta}_0^{(N)}=(N \widetilde{m}_0^{(N)}(k))_{k\in F^*} \in E^{(N)}$, for every $N\geq 0$. The existence of these two sequences can be proved via the law of large numbers. Now, for each $N\geq 0$ and $t\geq 0$, Theorem \ref{th:tl-general} establishes the existence of a coupling between $\eta^{(N)}_t$ and $\widetilde{\eta}^{(N)}_t$, where each of its components is generated by \eqref{eq:generatorFV}, with initial condition $(\eta_0^{(N)},\widetilde{\eta}_0^{(N)})$) which satisfies
$$
 N^{-1} \E\left[ d_1(\eta^{(N)}_t, \widetilde{\eta}^{(N)}_t)\right] \leq e^{- \rho t}d_{\text{\textrm{TV}}}\left(m_0^{(N)}, \widetilde{m}_0^{(N)} \right).
$$
Now let us prove that we can take the limit $N\rightarrow + \infty$. Since $F$ is countable and discrete, there exists an increasing sequence of finite sets $(F^*_n)_{n\geq0}$ such that $F^*= \cup_{n \geq 0} F^*_n$ and

\begin{align*}
d_{\textrm{TV}}(\mu T_t,\nu T_t)
&= \frac{1}{2} \sum_{k \in F^*} |\mu T_t \un_{\{ k \}} - \nu T_t \un_{\{k\}}| = \lim_{n\rightarrow + \infty} \frac{1}{2} \sum_{k \in F^*_n} |\mu T_t\un_{\{k\}} - \nu T_t\un_{\{k\}}|.
\end{align*}
The previous bound gives
$$
\E\left[\frac{1}{2} \sum_{k \in F^*_n} \left| \frac{\eta_t^{(N)}(k)}{N} - \frac{\widetilde{\eta}^{(N)}_t(k)}{N} \right|\right] \leq N^{-1} \E\left[ d_1(\eta^{(N)}_t, \widetilde{\eta}^{(N)}_t)\right] \leq e^{- \rho t}d_{\text{\textrm{TV}}}\left(m_0^{(N)}, \widetilde{m}_0^{(N)} \right).
$$
Using Theorem \ref{th:chaos} and taking the limit $N\rightarrow + \infty$, we find
\begin{align*}
\frac{1}{2} \sum_{k \in F^*_n} |\mu T_t\un_{\{k\}} - \nu T_t\un_{\{k\}}|
&\leq e^{- \rho t}d_{\text{\textrm{TV}}}\left(\mu, \nu \right).
\end{align*} 
Indeed, as we work in discrete space, the convergence in distribution is equivalent to that in total variation distance:
$$
\lim_{N \rightarrow + \infty} d_{\textrm{TV}} (m_0^{(N)}, \mu) = \lim_{N \rightarrow + \infty}  d_{\textrm{TV}} (\widetilde{m}_0^{(N)}, \nu)= 0.
$$
Furthermore all sequences in the expectations are increasing. Thus, taking the limit $n\rightarrow +\infty$, we obtain \eqref{eq:contractionT}. Finally, the existence of a QSD can be proved as in the proof of \cite[Theorem 1]{MMV12}. More precisely, let $\mu$ be any probability measure on $F^*$. We have, for all $s,t\geq 0$ such that $s\geq t$,
\begin{align*}
d_{TV}(\mu T_t, \mu T_s) 
&= d_{TV}(\mu T_t, \mu T_{s-t + t})= d_{TV}(\mu T_t, (\mu T_{s-t}) T_t) \leq e^{-\rho t}.
\end{align*}
Thus $(\mu T_t)_{t\geq0}$ is a Cauchy sequence for the total variation distance and thus admits a limit $\nu_{\textrm{qs}}$. This measure is then proved to be a QSD by standard arguments; see for instance \cite[Proposition 1]{MV12}.
\end{proof}

\begin{Rq}[Weaker assumptions]
\label{rq:cv-qsd}
Assumption $(A)$ is not necessary and even useless in the previous corollary. Indeed, using \cite[Theorem 1]{V11} and a similar argument of approximation, it's enough that the particle system does not explode. However, we used this proof for sake of completeness.
\end{Rq}

We can now proceed to the proof of the second corollary.
\begin{proof}[Proof of Corollary \ref{coro: IndTime}]
The proof is based on an "interpolation" between the bounds obtained in Corollary \ref{coro:CV-QSD} and Theorem \ref{th:chaos}.

Let us fix $t>0$, $u\in [0,1]$ and $\varphi$ a function such that $\Vert \varphi \Vert_{\infty} \leq 1$. By the Markov property, we have
\begin{align*}
\E_{\eta} [\vert m(\eta_t)(\varphi) - m(\eta) T_t \varphi\vert]
&\leq \E_{\eta} \left[ \vert m(\eta_t)(\varphi) - m(\eta_{tu}) T_{t(1-u)}\varphi\vert \right] \\
&\quad \ +  \E_{\eta}\left[\vert m(\eta_{tu}) T_{t(1-u)}\varphi - m(\eta) T_t\varphi\vert\right] \\
&\leq \sup_{\Vert \varphi \Vert_{\infty} \leq 1}\E_{\eta}\left[ \widetilde{\E}_{\eta_{tu}} \left[ \vert m(\widetilde{\eta}_{t(1-u)})(\varphi) - m(\eta_{tu}) T_{t(1-u)}\varphi \vert\right]\right] \\
&\quad \ + \E_{\eta}\left[d_{\textrm{TV}}(m(\eta_{tu}) T_{t(1-u)}, m(\eta) T_{ut} T_{t(1-u)})\right],
\end{align*}
where $(\widetilde{\eta}_t)_{t\geq 0}$ is a Markov process generated by $\eqref{eq:generator2}$ and where, for all $\eta \in E$, we denote by $\widetilde{\E}_{\eta}$ the conditional expectation of $(\widetilde{\eta}_t)_{t\geq 0}$ given the event $\{ \widetilde{\eta}_0 = \eta\}$. On the one hand, by Theorem \ref{th:chaos}, which is a uniform estimate on the initial condition, there exist $B, C >0$ such that
$$
\sup_{\Vert \varphi \Vert_{\infty} \leq 1} \tilde{\E}_{\eta_{tu}}\left[ \vert m(\widetilde{\eta}_{t(1-u)})(\varphi) - m(\eta_{tu}) T_{t(1-u)}\varphi\vert\right]\leq \dfrac{Ce^{Bt (1-u)}}{\sqrt{N-1}}. 
$$
On the other hand, from Corollary \ref{coro:CV-QSD}, we have
$$
\E_{\eta}\left[d_{\textrm{TV}}(m(\eta_{tu}) T_{t(1-u)}, m(\eta) T_{ut} T_{t(1-u)})\right] \leq e^{- \rho t(1-u)}.
$$
Choosing 
$$
u= 1 + \frac{1}{t(B+ \rho)} \log \left( \frac{BC}{\rho \sqrt{N-1}} \right),
$$
this gives
\begin{equation}
\label{cor: cst}
\sup_{\Vert \varphi \Vert_{\infty} \leq 1} \E_{\eta} \left[\vert m(\eta_t)(\varphi) - m(\eta) T_t\varphi\vert \right] \leq \frac{B + \rho}{B}\left( \frac{BC}{\rho \sqrt{N-1}} \right)^{\frac{\rho}{B+\rho}}.
\end{equation}

\end{proof}

\begin{Rq}[Weaker assumptions]
\label{rq:borne-unif}
We can weaken the assumption $\rho>0$ in the previous corollary. Indeed, it is enough to assume that there exist $C>0$ and $\lambda>0$ such that
$$
\forall t\geq 0, \ d_{\textrm{TV}}(\mu T_{t}, \nu T_{t})\leq C e^{-\lambda t}.
$$
Some sufficient conditions are given in \cite{DS67,MM02,MMV12}. We can also use a bound of convergence for the Fleming-Viot particle system as in Theorem \ref{th:tl-general}. In particular, when $F^*$ is finite, the particle system converges, uniformly in time, to the conditioned process; hence, if $\eta$ is distributed by the invariant distribution of the particle system (it exists since $E$ is finite) then it converges in law towards the quasi-stationary distribution.
\end{Rq}

\section{Complete graph dynamics}
\label{sect:GC}
In all this section, we study the example of a random walk on the complete graph. Let us fix $K\in \N^*$, $p>0$ and $N\in \N^*$, the dynamics of this example is as follows: we consider a model with $N$ particles and $K+1$ vertices $0,1,\dots, K$. The $N$ particles move on the $K$ vertices $1,\dots, K$ uniformly at random and jump to $0$ with rate $p$. When a particle reaches the node $0$, it jumps instantaneously over another particle chosen uniformly at random. This particle system corresponds to the model previously cited with parameters 
$$
Q_{i,j} = \frac{1}{K}, \quad \forall i,j \in F^{*}=\{1,\dots, K\}, i\neq j \ \text{and} \ p_0(i) = p, \quad \forall i\in F^{*}.
$$
The generator of the associated Fleming-Viot process is then given by
\begin{equation}
\label{eq:generator}
\mathcal{L}f(\eta) = \sum_{i=1}^{K} \eta(i) \left[  \sum_{j=1}^{K} (f(T_{i\rightarrow j} \eta)-f(\eta))\left( \dfrac{1}{K}+ p \dfrac{\eta(j)}{N-1} \right)\right],
\end{equation}
for every function $f$ and $\eta \in E$.

A process generated by \eqref{eq:generator} is an instance of inclusion processes studied in \cite{GRV10,GRV11,GRV12}. It is then related to models of heat conduction. One main point of \cite{GRV10,GRV11} is a criterion ensuring the existence and reversibility of an invariant distribution for the inclusion processes. In particular, they give an explicit formula of the invariant distribution of a process generated by \eqref{eq:generator}
 and we give this expression in Subsection \ref{subsec:invariant}. They also study different scaling limits which seem to be irrelevant for our problems.

Another application of this example comes from population genetics. Indeed,this model can also be referred as \textit{neutral evolution}, see for instance \cite{E12,W76}. More precisely, consider $N$ individuals possessing one type in $F^*=\{1, \dots, K\}$ at time $t$. Each pair of individuals
interacts at rate $p$. Upon an interacting event, one individual dies and the other one
reproduces. In addition, every individual changes its type (mutates) at rate $1$ and chooses uniformly at random a new
type in $F^*$. The measure $m(\eta_t)$ gives the proportions of types. The kind of mutation we consider here is often
referred to as parent-independent or the house-of-cards model.

In all this section, for any probability measure $\mu$ on $E$, we set in a classical manner $\E_{\mu}[ \cdot ] = \int_{F^*} \E_x[ \cdot ] \mu(dx)$ and $\p_{\mu} = \E_\mu[ \un_{\cdot}]$; similarly $\cov_\mu$ and $\var_\mu$ are defined with respect to $\E_\mu$.  

\subsection{The associated killed process}
We define the process $(X_t)_{t\geq 0}$ by setting
$$
X_t = \left\{
    \begin{array}{ll}
        Z_t & \mbox{if } t < \tau \\
        0 & \mbox{if } t \geq  \tau,
    \end{array}
\right.
$$
where $\tau$ is an exponential variable with mean $1/p$ and $(Z_t)_{t\geq 0}$ is the classical complete graph random walk (i.e. without extinction) on $\{1,\dots, K\}$. We have, for any bounded function $f$, 
$$
T_t f(x) = \E\left[ f(X_t) \ | \ X_0=x, X_t \neq 0 \right], \quad t\geq 0, x\in F^*.
$$
The conditional distribution of $X_t$ is simply given by the distribution of $Z_t$ :
$$
                         \p (X_t = i \ | \ X_t \neq 0) = \p (Z_t = i).
$$
The study of $(Z_t)_{t\geq 0}$ is trivial. Indeed, it converges exponentially fast to the uniform distribution $\pi_K$ on $\{1,\dots, K\}$. We deduce that for all $t  \geq 0$ and all initial distribution $\mu$, 
$$
d_{\textrm{TV}}(\mu T_{t}, \pi_{K}) = \sum_{i=1}^K \left|\p_{\mu} (X_t = i \ | \ \tau > t) - \pi_{K}(i)\right| \leq e^{- t}. 
$$
 
Thus in this case, the conditional distribution of $X$ converges exponentially fast to the Yaglom limit $\pi_K$.

\subsection{Correlations at fixed time}

The special form of $\mathcal{L}$, defined at \eqref{eq:generator}, makes the calculation of the two-particle correlations at fixed time easy. 
\begin{theo}[Two-particle correlations]
\label{th:cor-GC}
For all $k,l \in \lbrace 1, \dots , K\rbrace$, $k\neq l$ and any probability measure $\mu$ on $E$, we have for all $t\geq 0$
\begin{align*}
\cov_{\mu}(\eta_{t}(k),\eta_{t}(l))
& = \E_{\mu}\left[ \eta_{0}(k)\eta_{0}(l)\right] e^{-\frac{2K(N-1+p)}{K(N-1)}t} \\
&\quad + \dfrac{-N+1+2pN}{K(N-1+2p)} (\E_{\mu}\left[ \eta_{0}(k)\right]+\E_{\mu}\left[ \eta_{0}(l)\right])e^{-t}\\
&\quad - \E_{\mu}\left[ \eta_{0}(k)\right]\E_{\mu}\left[ \eta_{0}(l)\right]e^{-2t} + \dfrac{-N^{2}(p+1)+N}{K^{2}(N-1+p)}.
\end{align*}
\end{theo}

\begin{Rq}[Limit $t\rightarrow + \infty$]
By the previous theorem, we find for any probability measure $\mu$
\begin{align*}
\lim_{t\rightarrow +\infty} \cov_{\mu}(\eta_{t}(k),\eta_{t}(l)) 
&= \dfrac{-N^{2}(p+1)+N}{K^{2}(N-1+p)}= \cov(\eta(k),\eta(l)),
\end{align*}
where $\eta$ is distributed according to the invariant distribution; it exists since the state space is finite, see the next section.
\end{Rq}

\begin{Rq}[Limit $N\rightarrow + \infty$]

If  $\cov_{\mu}\left(\eta_{0}(k),\eta_{0}(l)\right)\neq 0$ then for all $k,l \in \lbrace 1, \dots , K\rbrace$, $k\neq l$ and any probability measure $\mu$, we have
\begin{align*}
 \cov_{\mu}\left(\dfrac{\eta_{t}(k)}{N},\dfrac{\eta_{t}(l)}{N}\right)
&\sim_{N}  e^{-2t } \cov_{\mu}\left(\dfrac{\eta_{0}(k)}{N},\dfrac{\eta_{0}(l)}{N}\right),
\end{align*}
where $u_{N}\sim_{N} v_{N}$ iff $\lim_{N\rightarrow +\infty} \dfrac{u_{N}}{v_{N}} = 1$.
\end{Rq}

\begin{proof}[Proof of Theorem \ref{th:cor-GC}]
For $k,l \in \{1,..,K\}$, let $\psi_{k,l}$ be the function $\eta\mapsto \eta(k)\eta(l)$. Applying the generator \eqref{eq:generator} to $\psi_{k,l}$ we obtain 
\begin{align*}
\mathcal{L}\psi_{k,l}(\eta)&= -\dfrac{2K(N-1+p)}{K(N-1)} \eta(k)\eta(l) + \dfrac{N-1}{K}(\eta(k)+\eta(l)).
\end{align*}
So, for all $t\geq0$,
$$\mathcal{L}\psi_{k,l}(\eta_{t}) = -\dfrac{2K(N-1+p)}{K(N-1)} \eta_{t}(k)\eta_{t}(l) + \dfrac{N-1}{K}(\eta_{t}(k)+\eta_{t}(l)).
$$
Using Kolmogorov's equation, we have
\begin{equation}
\label{eq:EDO}
\partial_{t} \E_{\mu}(\eta_{t}(k)\eta_{t}(l)) = -\dfrac{2K(N-1+p)}{K(N-1)} \E_{\mu}(\eta_{t}(k)\eta_{t}(l)) + \dfrac{N-1}{K}(\E_{\mu}(\eta_{t}(k))+\E_{\mu}(\eta_{t}(l))).
\end{equation}
Now if $\varphi_{k}(\eta) =\eta(k)$ then $\mathcal{L}\varphi_{k}(\eta) = \dfrac{N}{K} - \eta(k)$. We deduce that, for every $t\geq 0$,
$$
\partial_{t} \E_{\mu}(\eta_{t}(k)) = \dfrac{N}{K} - \E_{\mu}(\eta_{t}(k)) \ \text{ and } \ \E_{\mu}(\eta_{t}(k)) = \E_{\mu}(\eta_{0}(k))e^{-t} + \dfrac{N}{K}.
$$
Solving equation \eqref{eq:EDO} ends the proof.
\end{proof}

\subsection{Properties of the invariant measure}
\label{subsec:invariant} 
As $(\eta_t)_{t\geq 0}$ is an irreducible Markov chain on a finite state space, it is straightforward that it admits a unique invariant measure. In fact, this invariant distribution is reversible and we know its expression. 

\begin{theo}[Invariant distribution]
The process $(\eta_{t})_{t\geq 0}$ admits a unique invariant and reversible measure $\nu_N$, which is defined, for every $\eta \in E$, by
\begin{equation*}
\nu_{N}(\{\eta\})= Z^{-1} \prod_{i=1}^{K} \prod_{j=0}^{\eta(i)-1} \dfrac{N-1+Kpj}{j+1},
\end{equation*}
where $Z$ is a normalizing constant.
\end{theo}

This result was already proved in \cite[Section 4]{GRV10} and \cite[Theorem 2.1]{GRV11} but we give it for sake of completeness. 

\begin{proof}
A measure $\nu$ is reversible if and only if it satisfies the following balance equation
\begin{equation}
\label{eq:balance}
\nu(\{\eta\}) C(\eta,\xi) = \nu(\{\xi\}) C(\xi,\eta)\\
\end{equation}
where $ \xi = T_{i \rightarrow j} \eta$ and $C(\eta,\xi) = \mathcal{L} \un_{\xi} (\eta) = \eta(i) (K^{-1}+ p\eta(j)(N-1)^{-1})$.

Due to the geometry of the complete graph, it is natural to consider that $\nu$ has the following form 
$$
\nu(\{\eta\})= \frac{1}{Z}\prod_{i=1}^{K} l(\eta(i)),
$$ 
where $l: \{0, \dots, N\} \rightarrow [0,1]$ is a function and $Z$ is a normalizing constant. From \eqref{eq:balance}, we have
\begin{equation*}
l(\eta(i))l(\eta(j))\eta(i) (N-1+ Kp\eta(j)) = l(\eta(i)-1)l(\eta(j)+1)(\eta(j)+1) (N-1+ Kp(\eta(i)-1)),
\end{equation*}
for all $\eta \in E$ and $i,j \in \{1, \dots K \}$. Hence, 
$$
\frac{l(n)}{l(n-1)}\frac{n}{N-1 +Kp(n-1)} = \frac{l(m)}{l(m-1)}\frac{m}{N-1+Kp(m-1)} = u,
$$
for every $m,n\in\{ 1,\dots,N\}$ and some $u \in \R$. Finally,

\begin{eqnarray*}
\nu(\{\eta\})= \prod_{i=1}^{K} \left( u^{\eta(i)} \prod_{j=0}^{\eta(i)-1} \dfrac{N-1+Kpi}{i+1}  l(0)\right) =  l(0)^K u^{N} \prod_{i=1}^{K}  \prod_{j=0}^{\eta(i)-1} \dfrac{N-1+Kpj}{j+1},
\end{eqnarray*}
and $Z= 1/(l(0)^K u^{N})$.
\end{proof}

In particular, we have directly

\begin{coro}[Invariant distribution when $p=1/K$]
If $p=1/K$ then the process $(\eta_{t})_{t\geq 0}$ admits a unique invariant and reversible measure $\nu_N$, which is defined, for every $\eta \in E$, by
\begin{equation*}
\nu_{N}(\{\eta\})= Z^{-1} \prod_{i=1}^{K} \binom {N-2+\eta(i)} {N-2},
\end{equation*}
where $Z$ is a normalizing constant given by
$$
Z= \binom {(K+1)N-K-1} {KN-K-1}.
$$
\end{coro}

\begin{coro}[Marginal laws when $p=1/K$] 
If $p=1/K$ then for all $i\in \left\lbrace 1, \dots ,K\right\rbrace$ we have  
$$
\mathbb{P}_{\nu_{N}}(\eta(i)=x) = \dfrac{1}{Z} \binom {N-2+x} {N-2} \binom {KN-K-x} {(K-1)N-K},
$$
\end{coro}

\begin{proof}
Firstly let us recall the Vandermonde binomial convolution type formula: let $n,n_1, \dots, n_p$ be some non-negative integers satisfying $\sum_{i=1}^p n_i = n$,
we have
\begin{equation*}
\binom{r-1}{n-1}= \sum_{r_1 + \dots + r_p= r} \ \prod_{j=1}^p \binom{r_j-1}{n_j-1}.
\end{equation*}
The proof is based on the power series decomposition of $z \mapsto \left(z/(1-z)\right)^{n}=  \prod_{i=1}^p \left(z/(1-z)\right)^{n_{i}}.$
Using this formula, we find
\begin{align*}
\mathbb{P}_{\nu_{N}}(\eta(i)=x)
&= \sum_{\overline{x}\in E_{1}} \mathbb{P}_{\nu_{N}}(\eta=(x_{1}, \dots ,x_{i-1},x,x_{i+1} \dots , x_{K}))\\
&= \dfrac{1}{Z}\binom {N-2+x} {N-2} \sum_{\overline{x}\in E_{1}} \prod_{l=1}^{i-1}\prod_{l=i+1}^{K} \binom {N-2+x_{l}} {N-2}\\
&=\dfrac{1}{Z} \binom {N-2+x} {N-2} \binom {(K-1)(N-1)+N-x-1} {(K-1)(N-1)-1},\\
\end{align*}
where $E_{1} = \left\lbrace \overline{x}=(x_{1}, \dots ,x_{i-1},x_{i+1} \dots , x_{K}) | x_{1}+ \dots +x_{i-1}+x_{i+1} \dots + x_{K} = N-x \right\rbrace$.
\end{proof}

We are now able to express the particle correlations under this invariant measure.

\begin{theo}[Correlation estimates]
\label{theo: cor}
For all $i\neq j \in \lbrace 1, \dots, K\rbrace$, we have
\begin{equation*}
\vert \cov_{\nu_{N}}(\eta(i)/N,\eta(j)/N)\vert  \sim_{N} \dfrac{p+1}{K^{2}N},
\end{equation*}
 
\end{theo}

\begin{proof}
Let $\eta$ be a random variable with law $\nu_N$. As $\eta(1), \dots, \eta(K)$ are identically distributed and $\sum_{i=1}^{K} \eta(i) = N$ we have
$$
\cov_{\nu_N}(\eta(i)/N,\eta(j)/N) = -\dfrac{\var_{\nu_N} (\eta(i)/N)}{K-1}.
$$
Using the results of Section \ref{sect:spectre}, we have
\begin{align*}
\mathcal{L}(\eta(i)^{2}) 
&= \eta(i)^{2}\left[-2-\dfrac{2p}{N-1}\right] + \eta(i) \left[\dfrac{2N}{K} + \dfrac{2pN}{N-1} + \dfrac{K-2}{K}\right] + \dfrac{N}{K}.
\end{align*}
Using the fact that $\int \mathcal{L}(\eta(i)^{2}) d\nu_{N} = 0$ and $\int \eta(i) d\nu_{N} = \frac{N}{K}$, we deduce that
\begin{align*}
\int \eta(i)^{2} d\nu_{N} &= \dfrac{N \left[(2N+K-2)(N-1) + 2KNp + K(N-1)\right]} {2K^{2}(N-1+p)}.
\end{align*} 

Finally,
\begin{align*}
\var_{\nu_{N}}(\eta(i))
&= \int \eta(i)^{2} d\nu_{N} - \left(\int \eta(i) d\nu_{N}\right)^{2}=\dfrac{N(K-1)(Np+N-1)}{K^{2}(N-1+p)}, 
\end{align*}         
and thus, for $i\neq j$,  
\begin{align*}
\vert \cov_{\nu_{N}}(\eta(i)/N,\eta(j)/N)\vert &\sim_{N} \dfrac{p+1}{K^{2}N}.
\end{align*}               

\end{proof}

\begin{Rq}[Proof through coalescence methods]
Maybe we can use properties of Kingman's coalescent type process (which is a dual process) to recover some of our results (as for instance the previous correlation estimates). Indeed, after an interacting event, all individuals evolve independtly and it is enough to look when the first
mutation happens (backwards in time) on one of the genealogical tree branches. Nevertheless, we prefer to use another approach based on Markovian techniques. 
\end{Rq}

\begin{Rq}[Number of sites]
Theorem \ref{theo: cor} gives the rate of the decay of correlations with respect to the number of particles, but we also have a rate with respect to the number of sites $K$. For instance when $p=1/K$ and if $\eta$ is distributed under the invariant measure, then
\begin{align*}
\vert \cov_{\nu_{N}}(\eta(i)/N,\eta(j)/N)\vert &\sim_{K} \dfrac{1}{K(K-1)N}.
\end{align*}
\end{Rq}

The previous theorem shows that the occupation numbers of two distinct sites become non-correlated when the number of particles increases. In fact, Theorem \ref{theo: cor} leads to a propagation of chaos:

\begin{coro}[Convergence to the QSD]
We have
$$
\E_{\nu_{N}}\left[ d_{\textrm{TV}}(m(\eta),\pi_{K})\right] \leq \sqrt{\dfrac{K(p+1)}{N}},
$$
where $\pi_{K}$ is the uniform measure on $\{1, \dots, K\}$.
\end{coro}

\begin{proof} By the Cauchy-Schwarz inequality, we have
\begin{align*}
\E_{\nu_{N}}\left[\left| \dfrac{\eta(k)}{N} - \dfrac{1}{K}\right| \right]
&\leq \left( \E_{\nu_{N}} \left[\left| \dfrac{\eta(k)}{N} - \dfrac{1}{K}\right|^{2}\right]\right) ^{\frac{1}{2}} = \var_{\nu_{N}}\left(\dfrac{\eta(k)}{N}\right)^{1/2}\leq \sqrt{\dfrac{(K-1)(p+1)}{K^{2}N}}.
\end{align*}
Summing over $\{1, \dots, K\}$ ends the proof.
\end{proof}

The previous bound is better than the bound obtained in Theorem \ref{th:chaos} and its corollaries. This comes from the absence of bias term. Indeed, 
$$
\forall k \in F^*, \ \E_{\nu_{N}}[m(\eta)(k)] = \frac{1}{K} = \pi_K(k).
$$
The bad term in Theorem \ref{th:chaos} comes from, with the notations of its proof, the estimation of $|u_k(t)-v_k(t)|$ and Gronwall Lemma. 

\begin{Rq}[Parameters depending on $N$]
\label{rq:dependenceGC}
A nice application of explicit rates of convergence is to consider parameters depending on $N$. For instance, we can now consider that $p=p_N$ depends on $N$, this does not change neither the conditioned semi-goup nor the QSD but this changes the dynamics of our interacting-particle system. The last corollary gives that if $\lim_{N\rightarrow \infty} p_N/N =0$ then the empirical measure converges to the uniform measure.

\end{Rq}

\subsection{Long time behavior and spectral analysis of the generator}
\label{sect:spectre}

In this subsection, we point out the optimality of Theorem \ref{th:tl-general} in this special case. It gives
\begin{coro}[Wasserstein contraction]
\label{cor: WassGC}
For any processes $(\eta_{t})_{t>0}$ and $(\eta'_{t})_{t>0}$ generated by \eqref{eq:generator}, and
for any $t \geq 0$, we have
$$
\mathcal{W}_{d_1} (\mathcal{L}(\eta_t), \mathcal{L}(\eta'_t)) \leq e^{- t} \mathcal{W}_{d_1} (\mathcal{L}(\eta_0), \mathcal{L}(\eta'_0)).
$$
In particular, when $(\eta'_{0})$ follows the invariant distribution $\nu_{N}$ associated to \eqref{eq:generator}, we
get for every $t \geq 0$
$$
\mathcal{W}_{d_1} (\mathcal{L}(\eta_t), \nu_{N}) \leq e^{- t} \mathcal{W}_{d_1} (\mathcal{L}(\eta_0), \nu_{N}).
$$
\end{coro}

In particular, if $\lambda_1$ is the smallest positive eigenvalue of $-\mathcal{L}$, defined at \eqref{eq:generator}, then we have 
$$1=\rho \leq \lambda_1.$$
Indeed, on the one hand, let us recall that, as the invariant measure is reversible, $\lambda_1$ is the largest constant such that
\begin{equation}
\label{eq:var-zero}
\lim_{t\rightarrow + \infty} e^{2 \lambda t} \Vert R_t f - \nu_{N}(f) \Vert_{L^{2}(\nu_{N})}^{2} =0,
\end{equation}
for every $\lambda < \lambda_1$ and $f\in L^{2}(\nu_{N})$, where $(R_t)_{t\geq 0}$ is the semi-group generated by $\mathcal{L}$. See for instance \cite{B94,S97}. On the other hand, if $\lambda < 1$ then, by Theorem \ref{th:tl-general}, we have
\begin{align*}
e^{2 \lambda t} \Vert R_t f - \nu_{N}(f) \Vert_{L^{2}(\nu_{N})}^{2}
&=e^{2 \lambda t} \int_E \left( (\delta_\eta R_t) f - (\nu_{N} R_t) f\right)^2 \nu_{N}(d\eta)\\
&\leq 2 e^{2 \lambda t} \Vert f \Vert_\infty^2 \int_E  \mathcal{W}_{d_1} (\delta_\eta R_t, \nu_{N} R_t)^2 \nu_{N}(d\eta)\\
&\leq 2 e^{2 (\lambda-1) t} \Vert f \Vert_\infty^2 \int_E \mathcal{W}_{d_1} (\delta_\eta, \nu_{N})^2 \nu_{N}(d\eta),
\end{align*}
and then \eqref{eq:var-zero} holds. Now, the constant functions are trivially eigenvectors of $\mathcal{L}$ associated with the eigenvalue $0$, and if, for $k \in \{1,\dots,K\}$, $l \geq 1$ we set $\varphi^{(l)}_{k}:\eta\mapsto \eta(k)^{l}$ then the function $\varphi^{(1)}_k$ satisfies 
$$
\mathcal{L}\varphi^{(1)}_k= N/K - \varphi^{(1)}_k.
$$
In particular $\varphi^{(1)}_k - N/K$ is an eigenvector and $1$ is an eigenvalue of $-\mathcal{L}$. This gives $\lambda_1 \leq 1$ and finally $\lambda_1 =1$ is the smallest eigenvalue of $-\mathcal{L}$. By the reversibility, we have a Poincar\'{e} (or spectral gap) inequality
$$
\forall t \geq 0, \ \Vert R_t f - \nu_{N}(f) \Vert_{L^{2}(\nu_N)}^{2} \leq e^{-2 t} \Vert f - \nu_{N}(f) \Vert_{L^{2}(\nu_{N})}^{2}.
$$

\begin{Rq}[Complete graph random walk]
If $(a_i)_{1 \leq i \leq K}$ is a sequence such that $\sum_{i=1}^{K} a_{i}=0$ then the function $\sum_{i=1}^{K} \varphi^{(1)}_i$ is an eigenvector of $\mathcal{L}$. However, if $L$ is the generator of the classical complete graph random walk, $La = - a$ and then $a$ is also an eigenvector of $L$ with the same eigenvalue.
\end{Rq}

Let us finally give the following result on the spectrum of $\mathcal{L}$:

\begin{lem}[Spectrum of $-\mathcal{L}$]
The spectrum of $-\mathcal{L}$ is included in
$$
\left\{ \sum_{i=1}^{K} \lambda_{l_{i}} \ | \ l_1, \dots, l_K \in \{0, \dots, N\} \right\},
$$
where 
$$
\forall l \in \{0, \dots, N\}, \ \lambda_{l} = l + \dfrac{l(l-1)p}{N-1}.
$$
\end{lem}

\begin{proof}

For  every $ k \in \{1, \dots, K\}$ and $l\in \{0, \dots, N\}$, we have
\begin{align*}
\mathcal{L}\varphi^{(l)}_{k}(\eta) 
&= - \lambda_{l} \varphi_{k}^{(l)}(\eta) + Q_{l-1}(\eta),
\end{align*}
where $Q_{l-1}$ is a polynomial whose degree is less than $l-1$. A straightforward recurrence shows that whether there exists or not a polynomial function $\psi^{(l)}_{k}$, whose degree is $l$, satisfying $\mathcal{L}\psi^{(l)}_{k} = - \lambda_l \psi^{(l)}_{k}$ (namely $\psi^{(l)}_{k}$ is an eigenvector of $\mathcal{L}$). Indeed, it is possible to have $\psi^{(l)}_{k} = 0$ since the polynomial functions are not linearly independent ($F$ is finite). More generally, for all  $l_1, \dots, l_K \in \{1, \dots, N\}$, there exists a polynomial $Q$ with $K$ variables, whose degree with respect to the $i^{\text{th}}$ variable is strictly less than $l_i$, such that the function $\phi: \eta \mapsto \prod_{i=1}^{K} \eta(k_{i})^{l_{i}} + Q(\eta)$ satisfies 
$$
\mathcal{L}\phi = - \lambda \phi \ \text{where} \ \lambda  = \sum_{i=1}^{K} \lambda_{l_{i}}.
$$ 
Again, provided that $\phi \neq 0$, $\phi$ is an eigenvector and $\lambda$ an eigenvalue of $-\mathcal{L}$. Finally, as the state space is finite, using multivariate Lagrange polynomial, we can prove that every function is polynomial and thus we capture all the eigenvalues.
\end{proof}

\begin{Rq}[Cardinal of $E$]
As $\textrm{card}(F^*)=K$, we have 
$$
\textrm{card}(E) = \binom{N+K-1}{K-1} = \frac{(N+K-1)!}{N!(K-1)!}.
$$
In particular, the number of eigenvalues is finite and less than $\textrm{card}(E)$.
\end{Rq}

\begin{Rq}[Marginals]
For each $k$, the random process $(\eta_t(k))_{t\geq0}$, which is a marginal of a process generated by \eqref{eq:generator}, is a Markov process on $\N_{\textrm{N}}=\{0,\dots, N\}$ generated by 
\begin{align*}
\mathcal{G} f(x) 
&= (N-x) \left(\frac{1}{K}+\dfrac{px}{N-1}\right) (f(x+1) - f(x))\\
&\quad + x \left( \frac{K-1}{K} + \frac{p(N- x)}{N-1}\right)(f(x-1) - f(x)),
\end{align*}
for every function $f$ on $\N_{\textrm{N}}$ and $x\in \N_{\textrm{N}}$. We can express the spectrum of this generator. Indeed, let $\varphi_l: x \mapsto x^l$, for every $l\geq 0$. The family $(\varphi_l)_{0 \leq l\leq N}$ is linearly independent as can be checked with a Vandermonde determinant. This family generates the $L^2-$space  associated to the invariant measure since this space has a dimension equal to $N+1$. Now, similarly to the proof of the previous lemma, we can prove the existence of $N+1$ polynomials, which are eigenvectors and linearly independent, whose eigenvalues are $\lambda_0, \lambda_1, \dots, \lambda_N$.
\end{Rq}

\section{The two point space}
\label{sect:two}

We consider a Markov chain defined on the states $\left\lbrace 0, 1, 2\right\rbrace$ where $0$ is the absorbing state. Its infinitesimal generator $G$ is defined by
$$
G=\begin{bmatrix}
   0 & 0 & 0 \\
   p_0(1) & -a - p_0(1) & a\\
   p_0(2) & b & -b - p_0(b), \\
\end{bmatrix}
$$ 
where $a,b >0$, $p_0(1),p_0(2)\geq 0$ and $p_0(1)+p_0(2)>0$. The generator of the Fleming-Viot process with $N$ particles applied to bounded functions $f: E\rightarrow \mathbb{R}$ reads 
\begin{align}
\label{eq:generator2point} 
\mathcal{L}f(\eta) 
&= \eta(1) \left( a + p_{0}(1) \dfrac{\eta(2)}{N-1} \right)  (f(T_{1 \rightarrow 2} \eta)-f(\eta))\nonumber \\
&+ \eta(2) \left( b + p_{0}(2) \dfrac{\eta(1)}{N-1} \right)  (f(T_{2 \rightarrow 1} \eta)-f(\eta)).
\end{align}

\subsection{The associated killed process}

The long time behavior of the conditionned process is related to the eigenvalues and eigenvectors of the matrix:
$$
M=
\begin{bmatrix}
   -a-p_{0}(1) & a \\
   b & -b-p_{0}(2) 
\end{bmatrix}.
$$
Indeed see \cite[section 3.1]{MV12}. Its eigenvalues are given by      
\begin{equation*}
\lambda_{+} = \dfrac{-(a+b+p_{0}(1)+p_{0}(2)) +\sqrt{(a-b+p_{0}(1)- p_{0}(2))^{2}+4ab}}{2},
\end{equation*}

\begin{equation*}
\lambda_{-} = \dfrac{-(a+b+p_{0}(1)+p_{0}(2)) -\sqrt{(a-b+p_{0}(1)- p_{0}(2))^{2}+4ab}}{2},
\end{equation*}

and the corresponding eigenvectors are respectively given by
$$  v_{+} = \left (
   \begin{array}{c}
      a \\
      -A+\sqrt{A^{2}+4ab}\\
   \end{array}
   \right ) \ \text{ and } \ v_{-} = \left (
                                        \begin{array}{c}
                                          a \\
                                          -A-\sqrt{A^{2}+4ab}\\
                                        \end{array}
                                        \right ),$$
   
where $A= a-b+p_{0}(1)- p_{0}(2)$. Also set $\nu= v_+/(v_+(1) + v_+(2))$. From these properties, we deduce that

\begin{lem}[Convergence to the QSD]
There exists a constant $C>0$ such that for every initial distribution $\mu$, we have
$$
\forall t\geq 0, \ d_{\text{\textrm{TV}}}(\mu T_t, \nu) \leq C e^{-(\lambda_+ - \lambda_-) t}.
$$
\end{lem}
\begin{proof}
See \cite[Theorem 7]{MV12} and \cite[Remark 3]{MV12}.
\end{proof}
Note that
$$
\lambda_+ - \lambda_-
=\sqrt{(a+b)^2 +2 (a-b)(p_{0}(1)- p_{0}(2))+(p_{0}(1)- p_{0}(2))^{2}}> a+b - (\sup(p_0)-\inf(p_0))
$$
when $\sup(p_0)>\inf(p_0)$.

\subsection{Explicit formula of the invariant distribution}
\label{sect:two-inv} 
Firstly note that, as
$$
\forall \eta \in E, \eta(1) + \eta(2) =N,
$$
each marginal of $(\eta_t)_{t\geq 0}$ is a Markov process:

\begin{lem}[Markovian marginals]
The random process $(\eta_t(1))_{t\geq0}$, which is a marginal of a process generated by \eqref{eq:generator2point}, is a Markov process generated by  $\mathcal{G}$ defined by
\begin{align}
\mathcal{G} f(n) 
&= b_n (f(n+1) - f(n)) + d_n (f(n-1) - f(n)), \label{eq:gen-G}
\end{align}
for any function $f$ and $n \in \N_{\mathrm{N}}=\{0, \dots, N\}$, where
$$
b_n=(N-n) \left( b + p_{0}(2) \dfrac{n}{N-1} \right)\text{ and } d_n = n \left( a + p_{0}(1) \dfrac{N-n}{N-1} \right).
$$
\end{lem}

\begin{proof}
For every $\eta\in E$, we have $\eta=(\eta(1), N-\eta(1))$ thus the Markov property and the generator are easily deducible from the properties of $(\eta_t)_{t \geq 0}$. 
\end{proof}

From this result and the already known results on birth and death processes \cite{CJ12, C04}, we deduce that $(\eta_t(1))_{t\geq0}$ admits an invariant and reversible distribution $\pi$ given by
$$
\pi(n) = u_0 \prod_{k=1}^n \frac{b_{k-1}}{d_k} \ \text{ and } \ u_0^{-1}= 1 + \sum_{k=1}^N \frac{b_0 \cdots b_{k-1}}{d_{1} \cdots d_{k} },
$$ 
for every $n\in \N_{\mathrm{N}}$. This gives
$$
\pi(n) = u_0 \binom{N}{n} \prod_{k=1}^n \frac{b(N-1) + (k-1) p_0(2)}{ a(N-1) + (N-k) p_0(1)},
$$
and 
$$
u_0^{-1} = 1 + \prod_{k=1}^N \frac{b(N-1) + k p_0(2)}{ a(N-1) + k p_0(1)}.
$$
Similarly, as $\eta_t(2) = N-\eta_t(1)$, the process $(\eta_t(2))_{t\geq 0}$ is a Markov process whose invariant distribution is also easily calculable. The invariant law of $(\eta_t)_{t\geq 0}$, is then given by
$$
\nu_N((r_1,r_2))=\pi\left(\{ r_1\}\right),\quad \forall (r_1,r_2) \in E.$$
Note that if $p_0$ is not constant then we can not find a basis of orthogonal polynomials in the $L^2$ space associated to $\nu_N$. It is then very difficult to express the spectral gap or the decay rate of the correlations without using our main results.

\subsection{Rate of convergence}

Applying Theorem \ref{th:tl-general}, in this special case, we find:

\begin{coro}[Wasserstein contraction]
For any processes $(\eta_{t})_{t>0}$ and $(\eta'_{t})_{t>0}$ generated by \eqref{eq:generator2point}, and
for any $t \geq 0$, we have
\begin{align*}
\mathcal{W}_{d_1} (\mathcal{L}(\eta_t), \mathcal{L}(\eta'_t)) \leq e^{-\rho t} \mathcal{W}_{d_1} (\mathcal{L}(\eta_0), \mathcal{L}(\eta'_0)),
\end{align*}
where $\rho = a+b-(\sup(p_0)-\inf(p_0)).$ In particular, when $(\eta'_{0})$ follows the invariant distribution $\nu_N$ of \eqref{eq:generator2point}, we
get for every $t > 0$
$$
\mathcal{W}_{d_1} (\mathcal{L}(\eta_t), \nu_N) \leq e^{-\rho t} \mathcal{W}_{d_1} (\mathcal{L}(\eta_0), \nu_N).
$$
\end{coro}

This result is not optimal. Nevertheless, the error does not come from our coupling choice but it comes from how we estimate the distance. Indeed, this coupling induces a coupling between two processes generated by $\mathcal{G}$ defined by \eqref{eq:gen-G}. More precisely, let $\mathbb{L} = \mathbb{L}_Q + \mathbb{L}_{p}$ be the generator of our coupling introduced in the proof of Theorem \ref{th:tl-general} in this special case. We set $\mathbb{G} =\mathbb{G}_Q + \mathbb{G}_{p}$, where for any $n,n'\in \N_N$ and $f$ on $E \times E$,
$$
\mathbb{L}_Q f((n,N-n),(n',N-n'))= \mathbb{G}_Q \varphi_f(n,n'),
$$
$$
\mathbb{L}_p f((n,N-n),(n',N-n'))= \mathbb{G}_p \varphi_f(n,n'),
$$
and $\varphi_f(n,n') =f((n,N-n),(n',N-n'))$. It satisfies, for any function $f$ on $\N_{\mathrm{N}}$ and $n'> n$ two elements of $\N_{\mathrm{N}}$,

\begin{align*}
\mathbb{G}_{Q} f(n,n')
&= na \left( f(n-1,n'-1) - f(n,n') \right)\\
&+ (N-n')b \left( f(n+1,n'+1) - f(n,n') \right)\\
&+ (n'-n)b \left( f(n+1,n') - f(n,n') \right)\\
&+ (n'-n)a \left( f(n,n'-1) - f(n,n') \right),
\end{align*}
and
\begin{align*}
\mathbb{G}_{p} f(n,n')
&= p_0(1) \frac{n (N-n')}{N-1} \left( f(n-1,n'-1) - f(n,n') \right)\\
&+ p_0(2) \frac{n (N-n')}{N-1} \left( f(n+1,n'+1) - f(n,n') \right)\\
&+ p_0(1) \frac{n (n'-n)}{N-1} \left( f(n-1,n') - f(n,n') \right)\\
&+ p_0(2) \frac{(N-n')(n'-n)}{N-1} \left( f(n,n'+1) - f(n,n') \right)\\
&+ p_0(2) \frac{n(n'-n)}{N-1} \left( f(n+1,n') - f(n,n') \right)\\
&+ p_0(1) \frac{(N-n')(n'-n)}{N-1} \left( f(n,n'-1) - f(n,n') \right).
\end{align*}
Now, for any sequence of positive numbers $(u_k)_{k\in \{0, \dots, N-1\}}$, we introduce the distance $\delta_u$ defined by
$$
\delta_u (n,n') = \sum_{k=n}^{n'-1} u_k,
$$
for every $n,n' \in \N_{\mathrm{N}}$ such that $n'>n$. For all $n\in \N_{\mathrm{N}}\backslash \{N\}$, we have $\mathbb{G} \delta_u(n,n+1) \leq - \lambda_u \delta_u(n,n+1)$ where
$$
\lambda_u = \min_{k\in \{0, \dots, N-1\}} \left[ d_{k+1} - d_k \frac{u_{k-1}}{u_k} + b_k - b_{k+1} \frac{u_{k+1}}{u_{k}}\right],
$$ 
and thus, by linearity, $\mathbb{G} \delta_u(n,n') \leq - \lambda_u \delta_u(n,n')$, for every $n,n' \in \N_{\mathrm{N}}$. This implies that for any processes $(X_{t})_{t\geq 0}$ and $(X'_{t})_{t\geq 0}$ generated by $\mathcal{G}$ , and for any $t \geq 0$,
$$
\mathcal{W}_{\delta_u} (\mathcal{L}(X_t), \mathcal{L}(X'_t)) \leq e^{-\lambda_u t} \mathcal{W}_{\delta_u} (\mathcal{L}(X_0), \mathcal{L}(X'_0)).
$$
Note that, for every $n,n' \in \N_{\mathrm{N}}$, we have
$$
\min(u) d_1((n,N-n),(n',N-n')) \leq \delta_u(n,n') \leq \max(u) d_1((n,N-n),(n',N-n')),
$$
and then for any processes $(\eta_{t})_{t\geq 0}$ and $(\eta'_{t})_{t\geq 0}$ generated by \eqref{eq:generator2point}, and for any $t \geq 0$, we have
$$
\mathcal{W}_{d_1} (\mathcal{L}(\eta_t), \mathcal{L}(\eta'_t)) \leq \frac{\max(u)}{\min(u)} e^{-\lambda_u t} \mathcal{W}_{d_1} (\mathcal{L}(\eta_0), \mathcal{L}(\eta'_0)).
$$
Finally, using \cite[Theorem 9.25]{C04}, there exists a positive sequence $v$ such that $\lambda_v = \max_u \lambda_u>0$ is the spectral gap of the birth and death process $(\eta_t(1))_{t\geq 0}$. These parameters depend on $N$ and so we should write the previous inequality as
\begin{equation}
\label{eq:gap-2pts}
\mathcal{W}_{d_1} (\mathcal{L}(\eta_t), \mathcal{L}(\eta'_t)) \leq C(N) e^{-\lambda_N t} \mathcal{W}_{d_1} (\mathcal{L}(\eta_0), \mathcal{L}(\eta'_0)),
\end{equation} 
where $C(N)$ and $\lambda_N$ are two constants depending on $N$. In conclusion, the coupling introduced in Theorem \ref{th:tl-general} gives the optimal rate of convergence but we are not able to express a precise expression of $\lambda_N$ and $C(N)$. Nevertheless, in the section that follows, we will prove that, whatever the value of the parameters, the spectral gap is always bounded from below by a positive constant not depending on $N$.

\subsection{A lower bound for the spectral gap}
In this subsection, we study the evolution of $(\lambda_N)_{N\geq 0}$. Calculating $\lambda_N$ for small value of $N$ (it is the eigenvalue of a small matrix) and some different parameters show that, in general, this sequence is not monotone and seems to converge to $\lambda_+ - \lambda_-$. We are not able to prove this, but as it is trivial that for all $N\geq 0$, $\lambda_N>0$, we can hope that it is bounded from below. The aim of this section is to prove this fact.

Firstly, using similar arguments of subsection \ref{sect:spectre}, we have $\lambda_N \geq \rho$, for every $N\geq 0$. This result does not give us information in the case $\rho \leq 0$. However, we can use Hardy's inequalities \cite[Chapter 6]{Toulouse} and mimic some arguments of \cite{M99} to obtain:

\begin{theo}[A lower bound for the spectral gap]
\label{th:spectral gap2}
If $\rho \leq 0$ then there exists $c>0$ such that  
$$
\forall N \geq 0, \ \lambda_N > c.
$$  
\end{theo}

The rest of this subsection aims to prove this result. Hardy's inequalities are mainly based on the estimation of the quantities $B_{N,+}$ and $B_{N,-}$ be defined for every $i\in \N$ by
\begin{equation}
\label{eq:def-B}
 B_{N,+} (i) = \max_{x>i}\left( \sum_{y=i+1}^{x} \dfrac{1}{\pi(y) d_y}\right) \pi([x,N]),
\end{equation}
and
$$
 B_{N,-}(i) = \max_{x<i}\left( \sum_{y=x}^{i-1} \dfrac{1}{\pi(y) b_y}\right) \pi([1,x]).
$$

We recall that $\pi=\pi_N$ is the invariant distribution defined in Subsection \ref{sect:two-inv} and jumps rates $b$ and $d$ also depend on $N$.

 More precisely, \cite[Proposition 3]{M99} shows that if one wants to get a "good" lower bound of the spectral gap, one only needs to guess an "adequate choice" of $i$ and to apply the estimate
$$
\lambda_N \geq \frac{1}{4\max\{B_{N,+}(i) , B_{N,-}(i)\}}.
$$
So, we have to find an upper bound for these two quantities. Before to give it, let us prove that the invariant distribution $\pi$ is unimodal. Indeed, it will help us to choose an appropriate $i$.

\begin{lem}[Unimodality of $\pi$]
\label{lem: decreasing}
The sequence $(\pi(i+1)/\pi(i))_{i\geq 0}$ is decreasing. 
\end{lem}


\begin{proof}[Proof of Lemma \ref{lem: decreasing}]
For all $i \in \{1, \dots, N\}$, we set 
$$
g(i)= \dfrac{\pi(i+1)}{\pi(i)}=\dfrac{(N-i)(b(N-1)+ip_0(2))}{(i+1)((a+ p_0(1))(N-1)-ip_0(1))}.
$$
It follows that 

\begin{align*}
g(i+1)-g(i)  
&= \dfrac{\Lambda_N(i)}{(i+1)((a+ p_0(1))(N-1)-i p_0(1))(i+2)((a+ p_0(1))(N-1)-(i+1) p_0(1))}
\end{align*}

where 
\begin{align*}
\Lambda_N(i)
&= (N-i-1)(b(N-1)+(i+1)p_0(2))(i+1)((a+ p_0(1))(N-1)-i p_0(1)) \\
&\quad \ - (N-i)(b(N-1)+ip_0(2))(i+2)((a+ p_0(1))(N-1)-(i+1) p_0(1))\\
&= - \left[ b(N-1)- p_0(2) \right] \left[ (N+1) \left(a(N-1)- p_0(1) \right) + p_0(1) (N-i)(N-i-1) \right]\\
&\quad  \ - p_0(2) \left(i^{2} + 3i +2 \right) \left( a(N-1) - p_0(1) \right)\\
& \leq 0.
\end{align*}

We deduce the result.
\end{proof}


\begin{proof}[Proof of Theorem \ref{th:spectral gap2}]
Without less of generality, we assume that $p_0(1)\geq  p_0(2)$ and we recall that $\rho \leq 0$. We would like to know where $\pi$ reaches its maximum $i^*$ since it will be a good candidate to estimate $B_{N,+}(i^*)$ and $B_{N,-}(i^*)$. From the previous lemma, to find it, we look when $\pi(i+1)/\pi(i)$ is close to one. We have, for all $i\in \{ 1, \dots, N\}$,

\begin{align}
\dfrac{\pi(i+1)}{\pi(i)} 
&= \dfrac{b_i}{d_{i+1}}=1 + \dfrac{(p_0(1)-p_0(2))(i-i_1)(i-i_2)}{(i+1)\left((a+p_0(1))(N-1)-ip_0(1)\right)}\label{expression},
\end{align}
where $i_1$ and $i_2$ are the two real numbers given by 
$$
i_1 = \dfrac{ N (a+b+ p_0(1)-p_0(2)) - (a+b+ 2p_0(1)) - \sqrt{\Delta}}{2(p_0(1)-p_0(2))}
$$
and
$$
i_2 = \dfrac{ N (a+b+ p_0(1)-p_0(2)) - (a+b+ 2p_0(1)) + \sqrt{\Delta}}{2(p_0(1)-p_0(2))},
$$
where 
\begin{align*}
\Delta &=[N(a+b+p_0(1) - p_0(2)) - (a+b+ 2p_0(1))]^2 \\
&\quad  \ - 4 (N-1)(bN-a - p_0(1))(p_0(1)-p_0(2)).
\end{align*}
In particular,  $1\leq i_1 \leq N \leq i_2$. Furthermore, if $\lfloor . \rfloor$ denotes the integer part then
$$
\frac{\pi(\lfloor i_1 \rfloor + 2 )}{\pi(\lfloor i_1 \rfloor +1)}  \leq 1 \leq \frac{\pi(\lfloor i_1 \rfloor +1)}{\pi(\lfloor i_1 \rfloor )}.
$$

Let us define $m_N = \lfloor i_1 \rfloor +1$ and $l_N= 2(\lfloor \sqrt{N} \rfloor +1)$. Using a telescopic product, we have
$$
\frac{\pi(m_N+l_N)}{\pi(m_N)} = \frac{\pi(m_N+l_N-\lfloor \sqrt{N} \rfloor -1)}{\pi(m_N)} \prod_{j=1}^{\lfloor \sqrt{N} \rfloor +1} \frac{\pi(m_N+l_N-j + 1)}{\pi(m_N+l_N-j)},
$$
Using Lemma \ref{lem: decreasing} and the previous calculus, we have that the sequences $(\pi(i))_{i \geq m_N}$ and $(\pi(i+1)/\pi(i))_{i \geq 0}$ are decreasing and then
$$
\frac{\pi(m_N+l_N)}{\pi(m_N)} \leq \left( \frac{\pi(m_N+l_N- \lfloor \sqrt{N} \rfloor )}{\pi(m_N+l_N-\lfloor \sqrt{N} \rfloor  -1)} \right)^{\lfloor \sqrt{N} \rfloor +1}.
$$

Now using \eqref{expression} and some equivalents, there exists a constant $\delta_1>0$ (not depending on $N$) such that
\begin{align*}
\frac{\pi(m_N+l_N- \lfloor \sqrt{N} \rfloor )}{\pi(m_N+l_N-\lfloor \sqrt{N} \rfloor  -1)}
&\leq 1- \dfrac{\delta_1}{\sqrt{N}}.
\end{align*}
Using the fact that $1-x \leq e^{-x}$ for all $x \geq 0$, we finally obtain $\pi(m_N+l_N)/\pi(m_N) \leq e^{-\delta_1}$.
Similar arguments entail the existence of $\delta_2>0$ (also not depending on $N$) such that $\pi(m_N-l_N)/\pi(m_N) \leq e^{-\delta_2}$.
In conclusion, using Lemma \ref{lem: decreasing}, we have shown  that for all $i \geq m_N$ and $j \leq m_N$, the following inequalities holds:
$$
\pi(i+l_N) \leq e^{-\delta_1} \pi(i) \ \text{ and } \  \pi(j-l_N) \leq e^{-\delta_2} \pi(j).
$$
We are now armed to evaluate $B_{N,+}(m_N)$ defined in \eqref{eq:def-B}. Firstly, using the expressions of the death rate $d$ and $m_N$, there exist $\gamma>0$ (not depending on $N$) and $N_0 \geq 0$ such that for all $N \geq N_0$ and all $i \geq m_N+1 $, $d_i \geq \gamma N$. Let us fix $x\geq m_N+1$, using that $(\pi(i))_{i\geq m_N}$ is decreasing,  we have
\begin{align*}
\sum_{y=m_N+1}^{x} \frac{1}{\pi(y)}
&= \sum_{\{i,k | m_N +1 \leq k-i l_N \leq x\}}  \frac{1}{\pi(k-i l_N)} \\ 
& \leq \sum_{\{i,k | m_N +1 \leq k-i l_N \leq x\}} \frac{e^{-\delta_1 i}}{\pi(k)} \\
& \leq \dfrac{1}{1-  e^{-\delta_1 }}  \sum_{k=x-l_N +1}^{x} \dfrac{1}{\pi(k)}\\ 
& \leq \dfrac{l_N}{\pi(x)} \dfrac{1}{1-  e^{-\delta_1}}.
\end{align*}
Similarly, we have
$$
\pi([x,N]) =\sum_{\{k,i | x\leq k+i l _N \leq N \}} \un_{\{x+il_N \leq N\}} \Pi_N(k+il_N) \leq \frac{l_N \pi(x) }{1-  e^{-\delta_1}}.
$$
Using these three estimates, we deduce that, for every $N\geq N_0$, 
$$
B_{N,+}(m_N) \leq \frac{1}{\gamma N} \left( \dfrac{l_N}{1-  e^{-\delta_1}}\right)^{2} \leq \frac{1}{\gamma N} \left( \dfrac{2(\sqrt{N}+1)}{1-  e^{-\delta_1}}\right)^{2} \leq \frac{16}{\gamma (1-e^{-\delta_1})}. 
$$
The study of $B_{N,-}(m_N)$ is similar.
\end{proof}

\subsection{Correlations}

Using Theorem \ref{th:cor-gen}, we have

\begin{coro}[Correlations]
If $(\eta_t)_{t\geq 0}$ is a process generated by \eqref{eq:generator2point} then we have for all $t\geq 0$,
$$
\cov(\eta_{t}(k)/N,\eta_{t}(l)/N) \leq \frac{2}{N^2} \frac{1 - e^{- 2\rho t}}{\rho} \left(N (a \vee b) +  \sup(p_0) \frac{N^2}{N-1}\right).
$$
\end{coro}
If $\rho \leq 0$, the right-hand side of the previous inequality explodes as $t$ tends to infinity whereas these correlations are bounded by $1$. Nevertheless, using Theorem \ref{th:cor-gen}, Remark \ref{cor-amel} and Inequality \eqref{eq:gap-2pts}, we can prove that there exists two constants $C'(N)$, depending on $N$, and $K$, which does not depend on $N$, such that
$$
\sup_{t\geq 0} \cov(\eta_{t}(k)/N,\eta_{t}(l)/N) \leq C'(N)= \frac{K C(N)}{N \lambda_N},
$$
where $C(N)$ is defined in \eqref{eq:gap-2pts}. Even if Theorem \ref{th:spectral gap2} gives an estimate of $\lambda_N$, $C(N)$ is not (completely) explicit and we do not know if the right-hand side of the previous expression tends to $0$ as $N$ tends to infinity. This example shows the difficulty of finding explicit and optimal rates of the convergence towards equilibrium and the decay of correlations; it also illustrates that our main results are extremely useful when $\sup(p_0) \neq \inf(p_0)$.

\ \\ 
\textbf{Acknowledgement:} We would like to thank Amine Asselah and Djalil Chafa\"{i} for valuable discussions.
The work of B. Cloez was partially supported by the CIMI (Centre International de Math\'{e}matiques et d'Informatique)
 Excellence program while a postdoctoral scholar and by ANR MANEGE (09-BLAN-0215). M.-N. Thai was supported by grants from R\'{e}gion Ile-de-France.

\bibliographystyle{abbrv} 
\bibliography{ref}

\end{document}